\documentclass[letterpaper,12pt]{article}
\usepackage{amsmath,amssymb,amsfonts,epsfig,enumerate}
\usepackage{verbatim} 

\newtheorem{theo}{Theorem}[section]

\newtheorem{lemma}[theo]{Lemma}
\newtheorem{defn}[theo]{Definition}

\newtheorem{cor}[theo]{Corollary}

\newenvironment{proof}{\noindent {\sc Proof}.}
                {\phantom{a} \hfill \framebox[2.2mm]{ } \bigskip}

\newcommand{\NN}{\mathbb{N}}
\newcommand{\ZZ}{\mathbb{Z}}

\def\int{{\rm int}}

\newcommand{\T}{{\mathcal{T}}}

\renewcommand{\cal}{\mathcal}

\newcommand{\la}{\langle}
\newcommand{\ra}{\rangle}
\newcommand{\supp}{{\rm supp}}
\newcommand{\cG}{\mathcal{G}}
\newcommand\mult{\operatorname{mult}}

\title{Resolvable cycle decompositions \\
of complete multigraphs \\ and complete equipartite multigraphs \\
via layering and detachment}

\author{Amin Bahmanian\footnote{Campus Box 4520, Normal, Illinois, 61790-4520, USA.} \\
{\small  Illinois State University} \\ \\
Mateja \v{S}ajna\footnote{Department of Mathematics and Statistics, 585 King Edward Avenue, Ottawa, ON, K1N 6N5,Canada.}\\
{\small University of Ottawa}}

\begin{document}
\maketitle \baselineskip 17pt

\begin{abstract}
We construct new resolvable decompositions of complete  multigraphs and complete equipartite multigraphs into cycles of variable lengths (and a perfect matching if the vertex degrees are odd). We develop two techniques: {\em layering}, which allows us to obtain 2-factorizations of complete  multigraphs from existing 2-factorizations of complete graphs, and {\em detachment}, which allows us to construct resolvable cycle decompositions of complete equipartite multigraphs from existing resolvable cycle decompositions of complete multigraphs. These techniques are applied to obtain new 2-factorizations of a specified type for both complete multigraphs and complete equipartite multigraphs, with the emphasis on new solutions to the Oberwolfach Problem and the Hamilton-Waterloo Problem. In addition, we show existence of some  $\alpha$-resolvable cycle decompositions.

\end{abstract}

\newpage

\tableofcontents

\section{Introduction}

The graphs in this paper will be finite and undirected, often with parallel edges. Of particular interest are complete multigraphs and complete equipartite multigraphs; $\lambda K_n$ and $\lambda K_{n \times m}$ will denote the complete $\lambda$-fold graph with $n$ vertices and the complete $\lambda$-fold $n$-partite graph with $m$ vertices in each part, respectively.

A set $\{\cG_1,\ldots,\cG_t\}$ of subgraphs of a graph $\cG$ is called a {\em decomposition} of $\cG$ if $E(\cG)$ partitions into  $E(\cG_1),\ldots,E(\cG_t)$; if this is the case, we write $\cG=\cG_1 \oplus \ldots \oplus \cG_t$.
A {\em cycle decomposition} of $\cal G$ is a decomposition of $\cG$ into cycles, or into cycles and a single 1-factor in case every vertex in $\cal G$ is of odd degree. A cycle decomposition consisting of $k$ cycles of lengths $c_1,\dots,c_k$, respectively, is also called a {\em $(c_1,\ldots,c_k)$-cycle decomposition}.
A cycle decomposition of $\cal G$ is said to be {\em resolvable} if the set of cycles can be partitioned into subsets called {\em parallel classes} so that every vertex of $\cal G$ lies in exactly one cycle of each parallel class. A resolvable cycle decomposition can be viewed as a {\em 2-factorization}, that is, decomposition into {\em 2-factors} (spanning 2-regular subgraphs), or into 2-factors and a single 1-factor. This view is more convenient when we wish to specify not only cycle lengths, but also the types (isomorphism classes) of the 2-factors.

In 1980, Alspach \cite{Alspach-DM} asked whether the obvious necessary conditions for $K_n$ to admit a $(c_1,\ldots,c_k)$-cycle decomposition are also sufficient. In a recent breakthrough, this question was answered in the affirmative by Bryant et al \cite{MR3214677}. More recently still, Bryant et al \cite{MR3758242} extended this result to complete multigraphs, thus determining necessary and sufficient conditions for $\lambda K_n$ to admit a $(c_1,\ldots,c_k)$-cycle decomposition.

In contrast to general cycle decompositions, the problem of resolvable cycle decompositions of complete (multi)graphs is far from being completely solved. The focus in the literature is on 2-factorizations, where the types of the 2-factors, not merely cycle lengths, are specified. Thus, numerous particular 2-factorizations  of complete graphs are known to exist or not exist \cite{MR2246267}. By far the greatest number of results are known on  two special cases, the Oberwolfach Problem, where each 2-factor is of the same type, and the Hamilton-Waterloo Problem, where each 2-factor is of one of two possible types. The most prominent results to these two problems will be summarized in later sections, where needed.

In an earlier paper \cite{MR3487143}, the authors of the present paper developed a technique for constructing a $(c_1m,\ldots,c_km)$-cycle decomposition of $\lambda K_{n \times m}$ from a $(c_1,\ldots,c_k)$-cycle decomposition of $\lambda m K_n$. This method, detachment,  was applied to the decompositions guaranteed by \cite{MR3758242}, resulting in many new cycle decompositions of complete equipartite multigraphs. In the present paper, we seek to extend our earlier work to the resolvable case. To our knowledge, only one earlier paper \cite{MR1865547} used detachment to construct resolvable cycle decompositions.

In addition to detachment, which allows us to construct resolvable cycle decompositions of complete equipartite multigraphs from existing resolvable cycle decompositions of complete multigraphs, we develop second technique, layering, which allows us to obtain 2-factorizations of complete  multigraphs from existing 2-factorizations of complete (multi)graphs. The combination of these two techniques provides a general framework that can be applied to a wide variety of problems on resolvable cycle decomposition. We use it to construct new resolvable decompositions of complete  multigraphs and complete equipartite multigraphs into cycles of variable lengths, with the emphasis on new solutions to the  Oberwolfach and Hamilton-Waterloo Problems.

This paper is organized as follows. In Section~\ref{termsec}, we introduce the required terminology, and in Section \ref{detachsec}, the detachment technique is discussed. This is followed by an immediate application to $\alpha$-resolvability (a generalization of resolvability), which yields the only known results on $\alpha$-resolvable cycle decompositions of complete equipartite multigraphs to date. The layering technique is developed in Section~\ref{layerSec}, and new results on the  Oberwolfach Problem for complete equipartite multigraphs with cycles of variable lengths follow in Section~\ref{OPSec}. No results on the Hamilton-Waterloo Problem for $\lambda K_n$ with $\lambda \ge 2$ have been known to date; in Section~\ref{hwk_n}, the layering technique is further developed specifically for this problem, resulting in numerous solutions. In Section~\ref{hwK_nmultip}, applying detachment to these solutions, we generate related solutions to the Hamilton-Waterloo Problem for complete equipartite multigraphs. The layering technique is at its most powerful when used on bipartite 2-factorizations. In Section~\ref{sec:bip-Kn}, building upon the results from \cite{MR2833961}, we obtain the following results for bipartite 2-factors in complete multigraphs:
a complete solution to the Oberwolfach Problem, as well as an almost complete solution to the Hamilton-Waterloo Problem. Again, detachment is applied in the following section to obtain related results for complete equipartite multigraphs.
Finally, new results on 2-factorizations of specified types for complete multigraphs of small even order, as well as complete equipartite multigraphs with a small, even number of parts, are presented in Section~\ref{sec:new-gen}.
We conclude the paper with a brief mention of cycle frames.

\section{Terminology} \label{termsec}

The {\em multiset} $M$ consisting of $\mult_M(x)$ copies of each  element $x$ of a given set $S$ will be denoted by $[ x^{\la \mult_M(x) \ra}: x \in S ]$. The set $\supp(M)=\{ x \in M: \mult_M(x) \ge 1\}$ is called the {\em support} of $M$.
The {\em union} of multisets $M_1$ and $M_2$ is the multiset $M_1 \sqcup M_2=[ x^{\la \mult_{M_1}(x) + \mult_{M_2}(x) \ra}: x \in \supp(M_1) \cup \supp(M_2) ]$.

A {\em refinement} of a multiset $M=[ x_i: i=1,\ldots,s]$ of positive integers is any multiset $M'=[ y_j: j=1,\ldots,t]$ of positive integers such that  there exists a partition $\{ P_{i}: i=1,\ldots,s\}$ of $\{ 1,\ldots,t \}$ satisfying  $\sum_{j \in P_i} y_j=x_i$ for all $i \in \{1,\ldots,s\}$. For example, $[1^{\la 2 \ra}, 3^{\la 2 \ra}, 4,7,8]$ is a refinement of $[1^{\la 2 \ra}, 3, 7^{\la 2 \ra},8]$. For  an integer $n$ and a multiset $M$  we define $nM=[ nx: x \in M]$.

If $\cG$ is a graph, then  $\lambda \cG$ we denote the {\em $\lambda$-fold} graph $\cG$, obtained  by replacing each edge of $\cG$ with $\lambda$ parallel edges.

An {\em $r$-factor} of  $\cal G$ is a spanning $r$-regular subgraph of $\cal G$, and an {\em almost $r$-factor} (sometimes called a {\em near $r$-factor}) of  $\cG$ is an $r$-factor of $\cG - v$ for some vertex $v$ in $\cG$. An (almost) {\em $r$-factorization} of $\cal G$ is a decomposition  of $\cal G$ into (almost) $r$-factors. Additionally, when $\cG$ is regular of odd degree, we define a {\em 2-factorization} of  $\cal G$ as a decomposition of $\cal G$ into 2-factors and a single 1-factor.  A {\em holey $r$-factor} of $\lambda K_{n \times m}$ is an $r$-factor of $\lambda K_{n \times m}-W$ for some part $W$ of $\lambda K_{n \times m}$, and a {\em holey $r$-factorization}  of  $\lambda K_{n \times m}$ is a decomposition of $\lambda K_{n \times m}$ into holey $r$-factors.

A 2-regular graph is said to be of {\em type} $[c_1,\ldots,c_t]$ if it is a disjoint union of $t$ cycles of lengths $c_1,\ldots,c_t$, respectively, and a 2-factorization $\{ F_1,\ldots,F_k \}$ is of {\em type} $[T_1,\ldots,T_k]$ if the 2-factor $F_i$ is of type $T_i$, for all $i\in \{ 1,\ldots,k\}$. We abbreviate the  {\em uniform} types $[c,\ldots,c]$ and $[T,\ldots,T]$ to $[c^\ast]$ and $[T^\ast]$, respectively. A 2-factor type $T=[c_1,\ldots,c_t]$ is called {\em bipartite} if $c_1,\ldots,c_t$ are all even, and is {\em admissible} for $\cG$ if $\cG $ contains a 2-factor of type $T$. In particular, $T$ is admissible for $\lambda K_n$ if and only if $\sum_i c_i=n$, $c_i \ge 2$ for all $i$, and $c_i \ge 3$ for all $i$ if $\lambda=1$.  Similarly, a 2-factorization type $\mathcal T=[T_1,\ldots,T_k]$ is called {\em bipartite} if $T_1,\ldots,T_k$ are all bipartite 2-factor types, and is {\em admissible}  for an $r$-regular graph ${\cal G}$  if  $k=\lfloor \frac{r}{2} \rfloor$ and each $T_i$ is an admissible  2-factor type for ${\cal G}$.  The type of an  almost 2-factorization and holey 2-factorization  can be defined similarly.

A cycle decomposition ${\cal D}=\{ C_1,\ldots,C_k \}$ of $\cal G$ is called {\em $[\alpha_1,\ldots,\alpha_{\ell}]$-resolvable}  if $\cal D$ partitions into subsets ${\cal D}_1,\ldots,{\cal D}_{{\ell}}$ (called {\em parallel classes} of $\cal D$) such that each subgraph $\bigcup_{ C \in {\cal D}_j} C$ is a $2\alpha_j$-factor of $\cal G$; that is, every vertex of $\cal G$ lies in exactly $\alpha_j$ cycles in ${\cal D}_j$, for all $j \in \{1,\ldots,{\ell}\}$. Note that, necessarily, $\sum_{C \in {\cal D}_j} |E(C)|=\alpha_jn$ for all $j \in \{1,\ldots,{\ell}\}$, and $\cal G$ is regular of degree $\sum_{j=1}^{\ell} 2\alpha_j$ or $\sum_{j=1}^{\ell} 2\alpha_j +1$. An $[\alpha,\ldots,\alpha]$-resolvable cycle decomposition is called simply {\em $\alpha$-resolvable}, and a 1-resolvable cycle decomposition is also called {\em resolvable}.  Observe that a resolvable cycle decomposition is equivalent to a 2-factorization, and its parallel classes correspond to the 2-factors.

\section{Detachment} \label{detachsec}

The following result of \cite{MR3487143} will be needed in the development of our first technique, detachment. Here, $\mult_{\cG}(u,v)$ denotes the multiplicity of an edge $uv$ in a graph $\cG$, and for a given edge colouring of $\cG$, the symbol $\cG(i)$ will denote the spanning subgraph of $\cG$ containing all edges of colour $i$.

\begin{theo}\label{thm:Bah}\textup{ \cite[Theorem 1.4]{MR3487143}}
Assume there exists a $(c_1,\dots,c_k)$-cycle decomposition of $\lambda m K_n$. Then, for all $\ell=n,n+1, \ldots,mn$ there exist a graph $\cG_{\ell}=(V,E)$ of order $\ell$ and a function $g_{\ell}:V \rightarrow {\mathbb Z}^+$ with the following properties:
\begin{description}
\item[(P1)] $\cG_{\ell}$ is $n$-partite;
\item[(P2)] $\sum_{v\in W}g_{\ell}(v)=m$ for each part $W$ of $\mathcal G_{\ell}$;
\item[(P3)] $\mult_{\cG_{\ell}}(u,v)=\lambda g_{\ell}(u)g(v)$ for each pair of vertices $u,v$ from distinct parts of $\cG_{\ell}$;
\item[(P4)] $\cG_{\ell}$ admits a $k$-edge-colouring (if $\lambda m(n-1)$ is even) or a $(k+1)$-edge-colouring (if $\lambda m(n-1)$ is odd) such that for each each colour $i \in \{ 1, \ldots, k \}$:
\begin{description}
\item[(P4a)] colour class $\cG_{\ell}(i)$ has $c_im$ edges;
\item[(P4b)] $\deg_{{\cG_{\ell}}(i)}(v)\in\{0,2g_{\ell}(v)\}$ for each $v\in V$; and
\item[(P4c)] $\cG_{\ell}(i)$ has a unique non-trivial connected component.

\hspace{-10mm} In addition, if $\lambda m(n-1)$ is odd,
\item[(P4d)] colour class $\cG_{\ell}(k+1)$ has $\frac{1}{2}mn$ edges; and
\item[(P4e)] $\deg_{{\cG_{\ell}}(k+1)}(v)=g_{\ell}(v)$ for each $v\in V$.
\end{description}
\end{description}
\end{theo}

Now, we are ready to prove the following.

\begin{theo} \label{thm:main}
Let  $V(\lambda m K_n)=\ZZ_n$ and $V(\lambda K_{n \times m})=\cup_{i \in \ZZ_n} V_i$, where $V_0, V_1, \ldots, V_{n-1}$ are the parts of $\lambda K_{n \times m}$, each of size $m$. Suppose
${\cal D}=\{ C_1,\ldots,C_k \}$ is a cycle decomposition of $\lambda m K_n$.  Then there exists a cycle decomposition ${\cal D'}=\{ C_1',\ldots,C_k' \}$ of $\lambda K_{n \times m}$  such that $V(C_i')= \cup_{j \in V(C_i)} V_j$ for each $i \in \{ 1,\ldots, k\}$.
\end{theo}

\begin{proof}
As in the proof of Theorem~\ref{thm:Bah} \cite[Theorem 1.4]{MR3487143}, we first construct a graph ${\cal G}_n=(V,E)$ of order $n$ and a function $g_n:V \rightarrow {\mathbb Z}^+$ satisfying Properties (P1)--(P3) as follows: ${\cal G}_n \cong \lambda m^2K_n$ and $g_n(v)=m$ for all $v\in V$. Next, we take the $(c_1,\dots,c_k)$-cycle decomposition ${\cal D}$ of $\lambda m K_n$.
Replacing each edge in this decomposition by $m$ parallel edges we obtain a decomposition of ${\cal G}_n$ into $m$-fold cycles of lengths $c_1,\ldots, c_k$, plus an $m$-fold 1-factor if $\lambda m (n-1)$ is odd.  For $i=1,\ldots,k$, let
the $i$-th colour class ${\cal G}_n(i)$ be the $m$-fold cycle of length $c_i$ in this decomposition, together with the remaining $n-c_i$ isolated vertices. If $\lambda m (n-1)$ is odd, let ${\cal G}_n(k+1)$ be the $m$-fold 1-factor.
Property (P4) of Theorem~\ref{thm:Bah} then holds for ${\cal G}_n$ and $g_n$  with this edge colouring as well.

In the proof of Theorem~\ref{thm:Bah} \cite[Theorem 1.4]{MR3487143}, for $\ell=n,\ldots,mn-1$, a graph ${\cal G}_{\ell+1}$ is obtained from ${\cal G}_{\ell}$ by choosing a vertex $\alpha$ with $g_{\ell}(\alpha)>1$, adjoining a new vertex $\beta$, and replacing  certain edges $\alpha u$ (for some $u \in V$) with edges $\beta u$. Note that $\beta$ is placed in the same part of ${\cal G}_{\ell+1}$ as $\alpha$, so that all vertices in a part are offspring of the same vertex in ${\cal G}_{n}$.
In addition, the $k$-edge-colouring or $(k+1)$-edge-colouring
of ${\cal G}_{\ell+1}$ is inherited from the edge-colouring of $\cG_{\ell}$, and a suitable function $g_{\ell+1}$ is defined so that Properties (P1)--(P4) hold for ${\cal G}_{\ell+1}$
and $g_{\ell+1}$. A careful examination of the proof shows that Property (P4b) can be strengthened as follows:
\begin{description}
\item[(P4b)] $\deg_{{\cal G}_{\ell}(i)}(v)\in\{0,2g_{\ell}(v)\}$ for all $v\in V({\cal G}_{\ell})$;
\item[(P4b')] $\deg_{{\cal G}_{\ell}(i)}(\beta)=0$ if and only if $\deg_{{\cal G}_{\ell}(i)}(\alpha)=0$; and
\item[(P4b'')] if $\ell > n$, then for all $v \in V({\cal G}_{\ell})$ we have that  $\deg_{{\cal G}_{\ell}(i)}(v)=0$ if and only if $\deg_{{\cal G}_{\ell-1}(i)}(v)=0$.
\end{description}

For $\ell=mn$, Theorem~\ref{thm:Bah} then gives existence of an $n$-partite graph ${\cal G}_{mn}=(V,E)$ of order $mn$ and a function $g_{mn}:V \rightarrow {\mathbb Z}^+$ that satisfy Properties (P1)--(P4), including (P4b') and (P4b''). As in \cite[Theorem 1.4]{MR3487143}, it can be shown that ${\cal G}_{mn} \cong \lambda K_{n \times m}$ and $g_{mn}(v)=1$ for all $v \in V$. In case $\lambda m (n-1)$ is odd, colour class ${\cal G}_{mn}(k+1)$ is clearly a 1-factor of $\lambda K_{n \times m}$, while in both cases, the unique non-trivial connected component of each colour class ${\cal G}_{mn}(i)$ for $i \in \{ 1,\ldots,k \}$ is a 2-regular graph with $c_im$ edges, and hence a $c_im$-cycle; call it $C_i'$. Hence $\{ C_1',\ldots,C_k' \}$ is a $(c_1m,\dots,c_km)$-cycle decomposition of $\lambda K_{n \times m}$. It remains to show that for each $i \in \{ 1,\ldots, k\}$, we have $V(C_i')= \cup_{j \in V(C_i)} V_j$.

Take any $j \in \ZZ_n$ and any $u \in V_j$. By Property (P4b), $\deg_{{\cal G}_{mn}(i)}(u)\in\{0,2\}$, and by (P4b') and (P4b''), $\deg_{{\cal G}_{mn}(i)}(u)=0$ if and only if $\deg_{{\cal G}_{n}(i)}(j)=0$. Hence for all $i \in \{ 1,\ldots, k\}$, either $V_j \subseteq V(C_i')$ or $V_j \cap V(C_i') =\emptyset$, and the result follows.
\end{proof}

The following corollary will be among the most referred-to results of our paper. It describes how detachment applied to specific types of cycle decompositions of $\lambda m K_n$ results in specific types of cycle decompositions of $\lambda K_{n \times m}$.

\begin{cor}\label{cor:direct}
Let  ${\cal D}=\{ C_1,\ldots,C_k \}$ be a cycle decomposition of $\lambda m K_n$, and let ${\cal D'}=\{ C_1',\ldots,C_k' \}$ be a cycle decomposition of $\lambda K_{n \times m}$ whose existence is shown in Theorem~\ref{thm:main}. Then the following hold.
\begin{enumerate}[(i)]
\item If ${\cal D}$ is a $(c_1,\ldots,c_k)$-cycle decomposition, then ${\cal D'}$ is a $(c_1m,\ldots,c_km)$-cycle decomposition.
\item If ${\cal D}$ is $[\alpha_1,\ldots,\alpha_{\ell}]$-resolvable, then ${\cal D'}$ is $[\alpha_1,\ldots,\alpha_{\ell}]$-resolvable.
\item If ${\cal D}$ is equivalent to a 2-factorization of type $[T_1,\ldots,T_{\ell}]$, then  ${\cal D'}$ is equivalent to a 2-factorization of type $[mT_1,\ldots,mT_{\ell}]$.
\item If ${\cal D}$ is equivalent to an almost 2-factorization of type $[T_1,\ldots,T_{\ell}]$, then  ${\cal D'}$ is equivalent to a holey 2-factorization of type $[mT_1,\ldots,mT_{\ell}]$.
\end{enumerate}
\end{cor}

\begin{proof} As in Theorem~\ref{thm:main}, let  $V(\lambda m K_n)=\ZZ_n$ and $V(\lambda K_{n \times m})=\cup_{i \in \ZZ_n} V_i$, where $V_0, V_1, \ldots,$ $ V_{n-1}$ are the parts of $\lambda K_{n \times m}$, each of size $m$. Theorem~\ref{thm:main} then shows that $V(C_i')= \cup_{j \in V(C_i)} V_j$ for each $i \in \{ 1,\ldots, k\}$.
\begin{enumerate}[(i)]
\item It follows immediately that $C_i'$ is a $c_im$-cycle if $C_i$ is a $c_i$-cycle.
\item Assume ${\cal D}$ is $[\alpha_1,\ldots,\alpha_{\ell}]$-resolvable. Denote $K=\{1,\ldots,k \}$, and let ${\cal P}=\{ P_1,\ldots,P_{\ell}\}$ be a partition of $K$ such that for all $s \in \{1,\ldots,{\ell}\}$, we have that $\bigcup_{i \in P_s} C_i$ is a $2\alpha_s$-factor of $\lambda m K_n$. Fix $s \in \{1,\ldots,{\ell}\}$ and let ${\cal H}_s = \bigcup_{i \in P_s} C_i'$. It suffices to show that ${\cal H}_s$ is a $2\alpha_s$-factor of $\lambda K_{n \times m}$. Take any $j \in \ZZ_n$  and  any $u \in V_j$. By the assumption, vertex $j$ has degree $2\alpha_s$ in $\bigcup_{i \in P_s} C_i$, that is, it lies in $\alpha_s$ cycles in $\{ C_i:  i \in P_s\}$. By Theorem~\ref{thm:main}, for any $i$, we have that $j \in V(C_i)$ if and only if $u \in V(C_i')$. Hence $u$ lies in $\alpha_s$ cycles in $\{ C_i':  i \in P_s\}$, and ${\cal H}_s$ is a $2\alpha_s$-factor of $\lambda K_{n \times m}$.
    It follows that  ${\cal D'}$ is $[\alpha_1,\ldots,\alpha_{\ell}]$-resolvable.
\item Let $\{ F_1,\ldots,F_{\ell} \}$ be a 2-factorization of type $[T_1,\ldots,T_{\ell}]$ corresponding to ${\cal D}$. Take any $s \in \{ 1,\ldots, {\ell}\}$, and let $P_s \subset K$ be such that $F_s= \cup_{i \in P_s} C_i$. Let $F_s' = \bigcup_{i \in P_s} C_i'$. By the proof of (ii), $F_s'$ is a 2-factor of $\lambda K_{n \times m}$, and by the proof of (i), for each $i \in P_s$, $C_i'$ is a $c_im$-cycle if $C_i$ is a $c_i$-cycle.
    Hence, if $F_s$ is of type $T_s=[c_{s,1},\ldots,c_{s, t_s}]$, then $F_s'$ is of type $mT_s=[c_{s,1}m,\ldots,c_{s, t_s}m]$.
\item Let $\{ F_1,\ldots,F_{\ell} \}$ be an almost 2-factorization corresponding to ${\cal D}$. Take any $s \in \{ 1,\ldots,\ell \}$, and let $P_s \subset K$ be such that $F_s=\cup_{i \in P_s} C_i$. Moreover, let $j \in \ZZ_n$ be such that $F_s$ is a 2-factor of $\lambda m K_n - j$. By Theorem~\ref{thm:main}, $F_s'=\cup_{i \in P_s} C_i'$ is a spanning subgraph of $\lambda K_{n \times m} - V_j$, and as in (iii), it can be shown that if $F_s$ is of type $T_s$, then $F_s'$ is of type $mT_s$.
\end{enumerate}
\end{proof}

\section{$\alpha$-resolvable Cycle Decompositions} \label{resolmultip}

Before we move on to our second main technique, layering, we shall present an immediate application of detachment to obtain new $\alpha$-resolvable cycle decompositions.

Recall that a cycle decomposition of $\cG$ is  {\em $\alpha$-resolvable}  if the set of cycles can be partitioned into parallel classes so that every vertex of $\cal G$ lies in exactly $\alpha$ cycles of each parallel class.
Very little is known about existence of $\alpha$-resolvable cycle decompositions (of any graph!) for $\alpha \ge 2$. Corollory~\ref{cor:direct}(ii) yields the following general result.

\begin{theo}\label{thm:alpha}
If $\lambda m K_n$ admits an $\alpha$-resolvable $(c_1,\ldots,c_k)$-cycle decomposition, then $\lambda K_{n \times m}$ admits an $\alpha$-resolvable $(c_1m,\ldots,c_km)$-cycle decomposition.
\end{theo}
To exhibit the usefulness of Theorem \ref{thm:alpha}, we present two applications.

\begin{cor} \label{etares1}
Let $\alpha, \lambda, m, n$ be such that $3 \vert \alpha n$, $2\alpha \vert  \lambda m(n-1)$, and $4 \vert \lambda m$ if $(n,\alpha)=(6,1)$. Then  $\lambda K_{n\times m}$ admits an $\alpha$-resolvable $3m$-cycle decomposition.
\end{cor}

\begin{proof}
In \cite{MR1140809}, it is shown that $\lambda K_n$ with $\lambda (n-1)$ even admits an $\alpha$-resolvable 3-cycle decomposition if and only if $3 \vert \alpha n$, $2\alpha \vert  \lambda (n-1)$, and $4 \vert \lambda$ if $(n,\alpha)=(6,1)$. Hence, with our assumption, $\lambda m K_n$ admits an $\alpha$-resolvable 3-cycle decomposition, and the results follows from Theorem~\ref{thm:alpha}.
\end{proof}

\begin{cor} \label{etares2}
Let $\alpha, \lambda, m, n$ be such that $4 \vert \alpha n$ and $2\alpha \vert  \lambda m(n-1)$. Then  $\lambda K_{n\times m}$ admits an $\alpha$-resolvable $4m$-cycle decomposition.
\end{cor}

\begin{proof}
By \cite{MR2590230}, for each $\alpha\geq 1$, the obvious necessary conditions for the existence of an $\alpha$-resolvable 4-cycle decomposition of $\lambda K_n$ with $\lambda (n-1)$ even --- namely, $4 \vert \alpha n$ and  $2\alpha \vert  \lambda (n-1)$ --- are also sufficient. Our assumption then yields existence of an $\alpha$-resolvable 4-cycle decomposition of $\lambda m K_n$,
and applying Theorem~\ref{thm:alpha} completes the proof.
\end{proof}

\section{Layering} \label{layerSec}

Observe that if ${\cal G}$ is a regular graph of even degree, then we can easily obtain a 2-factorization of $\mu {\cal G}$  by taking $\mu$ copies of  a 2-factorization  of ${\cal G}$; we state this obvious fact more precisely in Lemma~\ref{lem:basic-layers} below.
When ${\cal G}$ is of odd degree, this approach does not work because of the additional 1-factor; however, for some graphs ${\cal G}$, the additional 1-factors in the $\mu$ layers of ${\cal G}$ can  be chosen to accommodate any bipartite 2-factor types. We shall refer to this technique as {\em layering}, and explain it more precisely for the complete multigraph in  Theorem~\ref{the:layers}, and for the complete equipartite multigraph in Corollary~\ref{cor:layers} and Theorem~\ref{the:layers2}.

The two basic ingredients of layering are described in the next two (rather obvious) lemmas.

\begin{lemma}\label{lem:basic-layers}
Let ${\cal G}$ be a $2k$-regular graph with a  2-factorization of type $[ T_1,\ldots,T_k]$. Then $\mu {\cal G}$ admits a 2-factorization of type $[ T_1^{\la \mu \ra},\ldots,T_k^{\la \mu \ra}]$.
\end{lemma}

\begin{lemma}\label{lem:layers}
Let ${\cal G}$ be the graph $K_n$ for $n$ even, or else $K_{n,n}$ for any $n$. Let $T$ be any admissible bipartite 2-factor type for $2{\cal G}$. Then $2 {\cal G}$ admits edge-disjoint 1-factors $I$ and $I'$ such that $I \oplus I'$ is a 2-factor of type $T$.
\end{lemma}

\begin{proof}
Clearly, $2{\cal G}$ admits a 2-factor $F$ of type $T$. Since every cycle of $F$ is of even length, $F=I \oplus I'$ for some edge-disjoint 1-factors $I$ and $I'$.
\end{proof}

\begin{theo}\label{the:layers}
For $i=1,\ldots,\ell$, assume $\mu_i K_n$ admits a 2-factorization of type ${\cal T}_i=[ T_{i,1},\ldots,$ $ T_{i,k_i}]$, with $k_i=\lfloor \frac{\mu_i(n-1)}{2} \rfloor$. Let ${\cal O}=\{ i: \mu_i(n-1) \mbox{ is odd} \}$, and  let $\mu$ be a positive integer.
Then $\mu K_n$ admits a 2-factorization of type
$$\sqcup_{i=1}^{\ell} [T_{i,1}^{\la x_i \ra},\ldots,T_{i,k_i}^{\la x_i \ra} ] \sqcup [T_1,\ldots,T_{\lfloor \frac{\beta}{2} \rfloor}]$$
for all $(x_1,\ldots,x_{\ell}) \in \NN^{\ell}$ such that $\sum_{i=1}^{\ell} x_i \mu_i=\mu$, $\beta=\sum_{i \in {\cal O}} x_i$, and all admissible bipartite 2-factor types $T_1,\ldots,T_{\lfloor \frac{\beta}{2} \rfloor}$ for $2K_n$.
\end{theo}

\begin{proof}
For $i=1,\ldots,\ell$, let ${\cal F}_i=\{ F_{i,1},\ldots, F_{i,k_i} \}$ be a 2-factorization of $\mu_i K_n$  of type ${\cal T}_i=[ T_{i,1},\ldots, T_{i,k_i}]$.
Take any $(x_1,\ldots,x_{\ell}) \in \NN^{\ell}$ such that $\sum_{i=1}^{\ell} x_i \mu_i=\mu$. Let $\beta=\sum_{i \in {\cal O}} x_i$, and note that $\lfloor \frac{\beta}{2} \rfloor$ is the number of bipartite 2-factors that we can fill in as desired. Hence take
any admissible bipartite 2-factor types $T_1,\ldots,T_{\lfloor \frac{\beta}{2} \rfloor}$ for $2K_n$.

For $i=1,\ldots,\ell$, let ${\cal G}_i^{(1)},\ldots,{\cal G}_i^{(x_i)}$ be the $x_i$ copies of $\mu_i K_n$ on the same vertex set, and let ${\cal F}_i^{(1)},\ldots,{\cal F}_i^{(x_i)}$, respectively, be their 2-factorizations of type ${\cal T}_i=[ T_{i,1},\ldots, T_{i,k_i}]$. Now $i \in {\cal O}$ if and only if ${\cal F}_i^{(j)}$, for every $j \in \{ 1,\ldots,x_i\}$, is also a 2-factorization of ${\cal G}_i^{(j)}-I_i^{(j)}$, for some 1-factor $I_i^{(j)}$ of ${\cal G}_i^{(j)}$. Since
$$| \{ I_i^{(j)}: i \in {\cal O}, j=1,\ldots,x_i \} |=\sum_{i \in {\cal O}} x_i=\beta,$$
we can relabel the 1-factors so that
$$\{ I_i^{(j)}: i \in {\cal O}, j=1,\ldots,x_i \} = \{ I_1,\ldots,I_{\beta} \}.$$
Furthermore, for each $k=1,\ldots, \lfloor \frac{\beta}{2} \rfloor$, we can relabel the vertices in the graph ${\cal G}_i^{(j)}$ corresponding to the 1-factor $I_{2k}$ so that $F_k=I_{2k-1} \oplus I_{2k}$ is a 2-factor of type $T_k$.

Let ${\cal H}=\oplus_{i \in {\cal O}} \oplus_{j=1}^{x_i} {\cal G}_i^{(j)}$ and ${\cal H}'=\oplus_{i \not\in {\cal O}} \oplus_{j=1}^{x_i} {\cal G}_i^{(j)}$, so that $\mu K_n= {\cal H} \oplus {\cal H}'$.
Observe that either ${\cal H}$ (if $\beta$ is even) or ${\cal H}-I_{\beta}$ (if $\beta$ is odd) decomposes as ${\cal H}_1 \oplus {\cal H}_2$, where
$${\cal H}_1=\oplus_{i \in {\cal O}} ( ({\cal G}_i^{(1)} - I_i^{(1)}) \oplus \ldots \oplus ({\cal G}_i^{(x_i)} - I_i^{(x_i)}) )$$
and
$${\cal H}_2= (I_1 \oplus I_2) \oplus \ldots \oplus (I_{2\lfloor \frac{\beta}{2} \rfloor -1} \oplus I_{2\lfloor \frac{\beta}{2} \rfloor}).$$
Now ${\cal H}_1$ admits a 2-factorization $\cup_{i \in {\cal O}} \cup_{j=1}^{x_i} {\cal F}_i^{(j)}$ of type $\sqcup_{i \in {\cal O}} [T_{i,1}^{\la x_i \ra},\ldots,T_{i,k_i}^{\la x_i \ra} ]$, and ${\cal H}_2$ admits a 2-factorization $\{ F_1,\ldots,F_{\lfloor \frac{\beta}{2} \rfloor} \}$ of type $[T_1,\ldots,T_{\lfloor \frac{\beta}{2} \rfloor}]$.

Finally, for each $i \not\in {\cal O}$, we have that ${\cal F}_i^{(1)} \cup \ldots \cup {\cal F}_i^{(x_i)}$ is a 2-factorization of $\oplus_{j=1}^{x_i} {\cal G}_i^{(j)}$ of type $[T_{i,1}^{\la x_i \ra},\ldots,T_{i,k_i}^{\la x_i \ra} ]$. Hence $\cup_{i \not\in {\cal O}} \cup_{j=1}^{x_i} {\cal F}_i^{(j)}$ is a 2-factorization of ${\cal H}$ of type $\sqcup_{i \not\in {\cal O}} [T_{i,1}^{\la x_i \ra},\ldots,T_{i,k_i}^{\la x_i \ra} ]$.

It follows that $\mu K_n= {\cal H}_1 \oplus {\cal H}_2 \oplus {\cal H}'$ admits a 2-factorization of type $\sqcup_{i=1}^{\ell} [T_{i,1}^{\la x_i \ra},\ldots,T_{i,k_i}^{\la x_i \ra} ] \sqcup [T_1,\ldots,T_{\lfloor \frac{\beta}{2} \rfloor}]$, as claimed.
\end{proof}

Next, we describe how to construct 2-factorizations of the complete equipartite multigraph from 2-factorizations of complete multigraphs by combining layering with detachment. Different results are obtained depending on whether layering (Corollary~\ref{cor:layers}) or detachment (Theorem~\ref{the:layers2}) is applied first. In the second case, additional 2-factor types with cycles of lengths not a multiple of $m$  can be obtained when $n$ is even and $m$ is odd.

\begin{cor}\label{cor:layers}
For $i=1,\ldots,\ell$, assume $\mu_i K_n$ admits a 2-factorization of type ${\cal T}_i=[ T_{i,1},\ldots,$ $ T_{i,k_i}]$,  for $k_i=\lfloor \frac{\mu_i(n-1)}{2} \rfloor$. Let ${\cal O}=\{ i: \mu_i (n-1) \mbox{ is odd} \}$, and let $\lambda$ and $m$ be positive integers.
Then $\lambda K_{n \times m}$ admits a 2-factorization of type
$$\sqcup_{i=1}^{\ell} [(mT_{i,1})^{\la x_i \ra},\ldots,(mT_{i,k_i})^{\la x_i \ra} ] \sqcup [mT_1,\ldots,mT_{\lfloor \frac{\beta}{2} \rfloor}]$$
for all $(x_1,\ldots,x_{\ell}) \in \NN^{\ell}$ such that $\sum_{i=1}^{\ell} x_i \mu_i=\lambda m$, $\beta=\sum_{i \in {\cal O}} x_i$, and all admissible bipartite 2-factor types $T_1,\ldots,T_{\lfloor \frac{\beta}{2} \rfloor}$ for $2K_n$.
\end{cor}

\begin{proof}
Using Theorem~\ref{the:layers}, we first obtain a 2-factorization of $\lambda m K_n$ of type
$$\sqcup_{i=1}^{\ell} [T_{i,1}^{\la x_i \ra},\ldots,T_{i,k_i}^{\la x_i \ra} ] \sqcup [T_1,\ldots,T_{\lfloor \frac{\beta}{2} \rfloor}].$$
 Corollary~\ref{cor:direct} then yields  a 2-factorization of $\lambda K_{n \times m}$ of type $\sqcup_{i=1}^{\ell} [(mT_{i,1})^{\la x_i \ra},\ldots,(mT_{i,k_i})^{\la x_i \ra} ] \sqcup [mT_1,\ldots,mT_{\lfloor \frac{\beta}{2} \rfloor}]$.
\end{proof}

\begin{theo}\label{the:layers2}
Let $n$ be even, and for $i=1,\ldots,\ell$, assume $\mu_i m K_n$ admits a 2-factorization of type ${\cal T}_i=[ T_{i,1},\ldots, T_{i,k_i}]$,  for $k_i=\lfloor \frac{\mu_im (n-1)}{2} \rfloor$. Let ${\cal O}=\{ i: \mu_i m(n-1) \mbox{ is odd} \}$, and let $\lambda$  be a positive integer.
Then $\lambda K_{n \times m}$ admits a 2-factorization of type
$$\sqcup_{i=1}^{\ell} [(mT_{i,1})^{\la x_i \ra},\ldots,(mT_{i,k_i})^{\la x_i \ra} ] \sqcup [T_1,\ldots,T_{\lfloor \frac{\beta}{2} \rfloor}]$$
for all $(x_1,\ldots,x_{\ell}) \in \NN^{\ell}$ such that $\sum_{i=1}^{\ell} x_i \mu_i=\lambda$, $\beta=\sum_{i \in {\cal O}} x_i$, and all admissible bipartite 2-factor types $T_1,\ldots,T_{\lfloor \frac{\beta}{2} \rfloor}$ for $2K_{n \times m}$ that are refinements of $[(2m)^{\la \frac{n}{2} \ra }]$.
\end{theo}

\begin{proof}
Take any $(x_1,\ldots,x_{\ell}) \in \NN^{\ell}$ such that $\sum_{i=1}^{\ell} x_i \mu_i=\lambda$. Let $\beta=\sum_{i \in {\cal O}} x_i$, and choose any admissible bipartite 2-factor types $T_1,\ldots,T_{\lfloor \frac{\beta}{2} \rfloor}$ for $2K_{n \times m}$ that are refinements of $[(2m)^{\la \frac{n}{2} \ra }]$. Note that, for each $s \in \{ 1, \ldots, \lfloor \frac{\beta}{2} \rfloor \}$, we can write $T_s=\sqcup_{k=1}^{\frac{n}{2}} T_{s,k}$, where each $T_{s,k}$ is a refinement of $[2m]$.

By Corollary~\ref{cor:direct}, for each $i \in \{ 1,\ldots,\ell \}$, the graph $\mu_i K_{n \times m}$ admits a 2-factorization of type $[ mT_{i,1},\ldots, mT_{i,k_i}]$.
For each $i \in \{ 1,\ldots, \ell \}$, take $x_i$ copies of $\mu_i K_{n \times m}$ and denote them by ${\cal G}_i^{(1)},\ldots,{\cal G}_i^{(x_i)}$. We shall assume that each of these graphs has the same vertex set $V_1 \cup \ldots \cup V_n$, with $V_1,\ldots,V_{n}$  its parts of size $m$. Let ${\cal F}_i^{(j)}$ be a 2-factorization of ${\cal G}_i^{(j)}$ of type $[ mT_{i,1},\ldots, mT_{i,k_i}]$.

Observe that $i \in {\cal O}$ if and only if the graphs ${\cal G}_i^{(j)}$, for all $j$, are of odd degree. For each $i \in {\cal O}$ and $j\in \{ 1,\ldots, x_i\}$, let $I_i^{(j)}$ be the 1-factor of ${\cal G}_i^{(j)}$ such that ${\cal F}_i^{(j)}$ is also a 2-factorization of ${\cal G}_i^{(j)}-I_i^{(j)}$. Since $\sum_{i \in {\cal O}} x_i=\beta$, we can relabel these 1-factors as
$$\{ I_i^{(j)}: i \in {\cal O}, j=1,\ldots,x_i \} = \{ I_1,\ldots,I_{\beta} \},$$
and the corresponding graphs as
$$\{ {\cal G}_i^{(j)}: i \in {\cal O}, j=1,\ldots,x_i \} = \{ {\cal G}_1,\ldots,{\cal G}_{\beta} \},$$
so that each $I_s$, for $s=1,\ldots,\beta$, is a 1-factor in ${\cal G}_s$. From the proof of Theorem~\ref{thm:main} it can be gleaned that each $I_s=\cup_{k=1}^{\frac{n}{2}} I_{s,k}$, where (after an appropriate relabeling of the parts) each $I_{s,k}$, for $k=1,\ldots,\frac{n}{2}$, is a 1-factor of ${\cal G}_s[V_{2k-1}\cup V_{2k}]$, which is isomorphic to $K_{m,m}$. Observe that the graphs $K_{m,m}-\tilde{I}$ are pairwise isomorphic for all 1-factors $\tilde{I}$ of $K_{m,m}$. This will allow us to arbitrarily relabel the vertices in each part $V_i$, for each graph ${\cal G}_s$.

Fix any $s \in \{ 1,\ldots,\lfloor \frac{\beta}{2} \rfloor \}$ and $k \in \{ 1,\ldots,\frac{n}{2} \}$. In ${\cal G}_{2s}$, relabel the vertices in part $V_{2k}$ so that $I_{2s-1,k} \oplus I_{2s,k}$ is a 2-factor of type $T_{s,k}$ in the graph $({\cal G}_{2s-1}\oplus {\cal G}_{2s})[V_{2k-1}\cup V_{2k}]$, which is isomorphic to $\eta_s K_{m,m}$ for some $\eta_s \ge 2$; this is possible by Lemma~\ref{lem:layers} since $T_{s,k}$ is a bipartite refinement of $[2m]$. Hence for each $s \in \{ 1, \ldots, \lfloor \frac{\beta}{2} \rfloor \}$, we have that $F_s=I_{2s-1} \oplus I_{2s}$ is a 2-factor of $\eta_s K_{n \times m}$ of type $T_{s}$.

Let ${\cal H}=\oplus_{i \in {\cal O}} \oplus_{j=1}^{x_i} {\cal G}_i^{(j)}=\oplus_{s=1}^{\beta} {\cal G}_s$ and ${\cal H}'=\oplus_{i \not\in {\cal O}} \oplus_{j=1}^{x_i} {\cal G}_i^{(j)}$, so that $\lambda K_{n \times m}= {\cal H} \oplus {\cal H}'$.
Observe that either ${\cal H}$ (if $\beta$ is even) or ${\cal H}-I_{\beta}$ (if $\beta$ is odd) decomposes as ${\cal H}_1 \oplus {\cal H}_2$, where
$${\cal H}_1=\oplus_{s=1}^{\beta} ({\cal G}_s - I_s) \quad \mbox{and} \quad{\cal H}_2= (I_1 \oplus I_2) \oplus \ldots \oplus (I_{2\lfloor \frac{\beta}{2} \rfloor -1} \oplus I_{2\lfloor \frac{\beta}{2} \rfloor}).$$
Now ${\cal H}_1$ admits a 2-factorization $\cup_{i \in {\cal O}} \cup_{j=1}^{x_i} {\cal F}_i^{(j)}$ of type $\sqcup_{i \in {\cal O}} [(mT_{i,1})^{\la x_i \ra},\ldots,(mT_{i,k_i})^{\la x_i \ra} ]$, and ${\cal H}_2$ admits a 2-factorization $\{ F_1,\ldots,F_{\lfloor \frac{\beta}{2} \rfloor} \}$ of type $[T_1,\ldots,T_{\lfloor \frac{\beta}{2} \rfloor}]$.

Finally, for each $i \not\in {\cal O}$, we have that ${\cal F}_i^{(1)} \cup \ldots \cup {\cal F}_i^{(x_i)}$ is a 2-factorization of $\oplus_{j=1}^{x_i} {\cal G}_i^{(j)}$ of type $[(mT_{i,1})^{\la x_i \ra},\ldots,(mT_{i,k_i})^{\la x_i \ra} ]$. Hence $\cup_{i \not\in {\cal O}} \cup_{j=1}^{x_i} {\cal F}_i^{(j)}$ is a 2-factorization of ${\cal H}$ of type $\sqcup_{i \not\in {\cal O}} [(mT_{i,1})^{\la x_i \ra},\ldots,(mT_{i,k_i})^{\la x_i \ra} ]$.

It follows that, as claimed, $\lambda K_{n \times m}= {\cal H}_1 \oplus {\cal H}_2 \oplus {\cal H}'$ admits a 2-factorization of type $\sqcup_{i=1}^{\ell} [(mT_{i,1})^{\la x_i \ra},\ldots,$ $(mT_{i,k_i})^{\la x_i \ra} ] \sqcup [T_1,\ldots,T_{\lfloor \frac{\beta}{2} \rfloor}]$.
 \end{proof}

\section{The Oberwolfach Problem for complete equipartite multigraphs } \label{OPSec}

In this section, we use Corollary~\ref{cor:direct}, as well as basic layering, to obtain concrete new results on the Oberwolfach Problem for complete equipartite multigraphs.
The Oberwolfach Problem $OP({\cal G}; T)$ for the graph ${\cal G}$ and admissible 2-factor type $T$ asks whether ${\cal G}$ admits a 2-factorization  of type $[T^\ast]$. If $T=[m_1,\ldots,m_t]$, we also write $OP({\cal G};m_1,\ldots,m_t)$ instead of $OP({\cal G}; T)$.
Corollary~\ref{cor:direct}(iii)  gives the following general result.

\begin{theo}\label{thm:OP}
If $OP(\lambda m K_n; T)$ has a solution, then $OP(\lambda K_{n \times m}; mT)$ has a solution.
\end{theo}
For a uniform 2-factor type $T'$, $OP(\lambda K_{n \times m}; T')$  has been completely solved  \cite{MR2028231}, so Theorem~\ref{thm:OP} yields no new results when $T$ is uniform, however, it leads to  simpler proofs of existing results.

A large number of specific cases of the Oberwolfach Problem with non-uniform factors are known to have solutions (see \cite{MR2246267}), however, we shall  limit our application of Theorem~\ref{thm:OP} to three of the most comprehensive results. The case of a bipartite 2-factor type is postponed to Section~\ref{sec:bip-Kn}.

In our first example, we consider the case where the number parts in $\lambda K_{n \times m}$ is small.

\begin{cor}
Let $n$ be odd,  $n \le 40$, and let $T$ be an admissible 2-factor type for $K_n$ such that $T \not\in \{ [4,5],[3^{\la 2 \ra}, 5] \}$. Then $OP(\lambda K_{n \times m}; mT)$  has a solution.
\end{cor}

\begin{proof}
Papers \cite{MR1768284,MR2100737,MR2239309,MR2675892} jointly prove that $OP(K_n; T)$ has a solution whenever $n \le 40$ and $T \not\in \{ [3^{\la 2 \ra}],[3^{\la 4 \ra}],[4,5],[3^{\la 2 \ra}, 5] \}$. With our assumptions, Lemma~\ref{lem:basic-layers} implies that $OP(\lambda mK_n; T)$ has a solution, and hence $OP(\lambda K_{n \times m}; mT)$  has a solution by Theorem~\ref{thm:OP}.
\end{proof}

In our second example, we consider 2-factors consisting of two cycles.

\begin{cor}
Let $\lambda, m, n, r,s$ be positive integers such that $r,s \ge 3$, $n=r+s$,  $(\lambda m, [r,s]) \ne (1,[4,5])$, and $\lambda m \equiv 0 \pmod{4}$ if $[r,s]=[3,3]$. Then $OP(\lambda K_{n \times m}; rm,sm)$ has a solution.
\end{cor}
\begin{proof}
In \cite{MR3033656}  it is shown that for $r,s\geq 3$ and $n=r+s$,  $OP(\lambda K_n;r,s)$  has a solution  except for $OP(\lambda K_6;3,3)$ with $\lambda  \not \equiv 0 \pmod 4$, and $OP(K_9;4,5)$, none of which has a solution. Hence, with our assumptions, $OP(\lambda mK_{n}; r,s)$ has a solution, and hence $OP(\lambda K_{n \times m}; rm,sm)$ has a solution by
Theorem~\ref{thm:OP}.
\end{proof}

Our last example is perhaps the most interesting of the three.

\begin{cor}
Let $\lambda$ and $m$ be positive integers. There are infinitely many primes $n$ such that $OP(\lambda K_{n \times m};mT)$ has a solution for any admissible 2-factor type $T$ of $K_n$.
\end{cor}

\begin{proof}
By \cite{MR2558441}, there are infinitely many primes $n$  such that $OP(K_n;T)$ has a solution for all admissible 2-factor types $T$. By Lemma~\ref{lem:basic-layers}, the same statement holds for $\lambda mK_n$, and hence the results follows by
 Theorem~\ref{thm:OP}.
\end{proof}

\section{The Hamilton-Waterloo Problem for complete multigraphs} \label{hwk_n}

The Hamilton-Waterloo Problem $HWP({\cal G}; T_1,T_2;\alpha_1,\alpha_2)$ for the $r$-regular graph ${\cal G}$, 2-factor types $T_1$ and $T_2$, and non-negative integers $\alpha_1$ and $\alpha_2$ such that $\alpha_1+\alpha_2=\lfloor \frac{r}{2} \rfloor$ asks whether ${\cal G}$ admits a 2-factorization of type $[T_1^{\la \alpha_1 \ra},T_2^{\la \alpha_2 \ra}]$. Note that $HWP({\cal G}; T_1,T_2;\lfloor \frac{r}{2} \rfloor,0)$ is simply $OP({\cal G}; T_1)$.

To our knowledge, no results on the Hamilton-Waterloo Problem for $\lambda K_n$ with $\lambda \ge 2$ are known to date.
In this section, we utilize our layering technique to construct  solutions to the Hamilton-Waterloo Problem for complete multigraphs from known solutions to the Hamilton-Waterloo Problem for complete graphs. We shall begin with a general result, and follow with concrete applications.

\begin{theo}\label{the:layers-HWP}
For $i=1,\ldots,\ell$, assume $HWP(\mu_i K_n; \bar{T}_1,\bar{T}_2;\alpha_i,\gamma_i)$ has a solution. Let ${\cal O}=\{ i: \mu_i(n-1) \mbox{ is odd} \}$, and  let $\lambda$ be a positive integer. Take any $(x_1,\ldots,x_{\ell}) \in \NN^{\ell}$ such that $\sum_{i=1}^{\ell} x_i \mu_i=\lambda$, and let $\beta=\sum_{i \in {\cal O}} x_i$.
\begin{enumerate}[(i)]
\item If $\beta \le 1$, then $HWP(\lambda K_n; \bar{T}_1,\bar{T}_2;\sum_{i=1}^{\ell} x_i\alpha_i,\sum_{i=1}^{\ell} x_i\gamma_i)$ has a solution.
\item If $\bar{T}_1$ is bipartite, then $HWP(\lambda K_n; \bar{T}_1,\bar{T}_2;\sum_{i=1}^{\ell} x_i\alpha_i+\bar{\alpha},\sum_{i=1}^{\ell} x_i\gamma_i+\bar{\gamma})$ has a solution for all nonnegative integers $\bar{\alpha}, \bar{\gamma}$ such that
\begin{enumerate}[(a)]
\item $\bar{\alpha}+\bar{\gamma}=\lfloor \frac{\beta}{2} \rfloor$ and
\item $\bar{\gamma}=0$ if $\bar{T}_2$ is not bipartite.
\end{enumerate}
\end{enumerate}
\end{theo}

\begin{proof}
Assume $\mu_i K_n$ admits a 2-factorization of type ${\cal T}_i=[ \bar{T}_1^{\la \alpha_i\ra}, \bar{T}_2^{\la \gamma_i\ra} ]$,
for all $i=1,\ldots,\ell$. Take any $(x_1,\ldots,x_{\ell}) \in \NN^{\ell}$ such that $\sum_{i=1}^{\ell} x_i \mu_i=\lambda$, and let $\beta=\sum_{i \in {\cal O}} x_i$.
\begin{enumerate}[(i)]
\item If $\beta \le 1$, then $\lfloor \frac{\beta}{2} \rfloor =0$, and by Theorem~\ref{the:layers}, $\lambda K_n$ admits a 2-factorization of type
    $ \sqcup_{i=1}^{\ell} [ \bar{T}_1^{\la x_i \alpha_i\ra}, \bar{T}_2^{\la x_i \gamma_i\ra} ]=[ \bar{T}_1^{\la \sum_{i=1}^{\ell} x_i \alpha_i\ra}, \bar{T}_2^{ \la \sum_{i=1}^{\ell} x_i \gamma_i\ra} ]$.

\item Assume $\bar{T}_1$ is bipartite, and let $\bar{\alpha}, \bar{\gamma}$ be nonnegative integers satisfying (a) and (b). Now we use Theorem~\ref{the:layers} with $[T_1,\ldots,T_{\lfloor \frac{\beta}{2} \rfloor}]=[\bar{T}_1^{\la \bar{\alpha} \ra}, \bar{T}_2^{\la \bar{\gamma} \ra}]$ to show that $\lambda K_n$ admits a 2-factorization of type
    $ \sqcup_{i=1}^{\ell} [ \bar{T}_1^{\la x_i \alpha_i\ra}, \bar{T}_2^{\la x_i \gamma_i\ra} ] \sqcup [\bar{T}_1^{\la \bar{\alpha} \ra}, \bar{T}_2^{\la \bar{\gamma} \ra}]=[ \bar{T}_1^{\la \sum_{i=1}^{\ell} x_i \alpha_i +\bar{\alpha} \ra}, \bar{T}_2^{\la \sum_{i=1}^{\ell} x_i \gamma_i+ \bar{\gamma}\ra} ]$.
\end{enumerate}
In both cases, the result follows immediately.
\end{proof}

In practice, the following corollary of Theorem~\ref{the:layers-HWP} will be more convenient to use for $n$ odd.

\begin{cor}\label{cor:layers-HWP-odd}
Let $T_1$ and $T_2$ be admissible 2-factor types for $K_n$ with $n$ odd, and let $\lambda \ge 2$.
\begin{enumerate}[(i)]
\item If $HWP(K_n;T_1,T_2;\frac{n-1}{2},0)$, $HWP(K_n;T_1,T_2;0,\frac{n-1}{2})$, and $HWP(K_n;T_1,T_2;\alpha,\frac{n-1}{2}-\alpha)$ all have solutions, then $HWP(\lambda K_n;T_1,T_2;\alpha',\lambda \frac{n-1}{2}-\alpha')$ has a solution for all $\alpha' \in \{ 0,\ldots,\lambda \frac{n-1}{2} \}$ such that $\alpha' \equiv \alpha \pmod{\frac{n-1}{2}}$.
\item If $HWP(K_n;T_1,T_2;\alpha,\frac{n-1}{2}-\alpha)$ has a solution for all $\alpha \in \{ 0,\ldots,\frac{n-1}{2} \}$, then \\ $HWP(\lambda K_n;T_1,T_2;\alpha',\lambda \frac{n-1}{2}-\alpha')$ has a solution for all $\alpha' \in \{ 0,\ldots,\lambda \frac{n-1}{2} \}$.
\end{enumerate}
\end{cor}

\begin{proof}
\begin{enumerate}[(i)]
\item Take any $\alpha' \in \{ 0,\ldots,\lambda \frac{n-1}{2} \}$ such that $\alpha' \equiv \alpha \pmod{\frac{n-1}{2}}$. We apply Theorem~\ref{the:layers-HWP}(i) with all $\mu_i=1$. Since $n$ is odd, we have ${\cal O}=\emptyset$ and $\beta=0$.  It suffices to find non-negative integers $\alpha_1,\ldots,\alpha_{\ell}$ and $x_1,\ldots,x_{\ell}$ such that $\lambda=\sum_{i=1}^{\ell} x_i$, $\alpha'= \sum_{i=1}^{\ell} x_i\alpha_i$, and $HWP(K_n;T_1,T_2;\alpha_i,\frac{n-1}{2}-\alpha_i)$ has a solution for all $i$.

Write $\alpha'=q \frac{n-1}{2}+\alpha$ so that $0 \le q \le \lambda$. If $q=\lambda$, then $\alpha=0$, and we let $\ell=1$, $\alpha_1=\frac{n-1}{2}$, $x_1=q$. Hence we may assume $q < \lambda$. In this case, we let $\ell=3$, $\alpha_1=\frac{n-1}{2}$, $\alpha_2=\alpha$, $\alpha_3=0$, $x_1=q$, $x_2=1$, and $x_3=\lambda-q-1$. Since $HWP(K_n;T_1,T_2;\frac{n-1}{2},0)$, $HWP(K_n;T_1,T_2;0,\frac{n-1}{2})$, and $HWP(K_n;T_1,T_2;\alpha,\frac{n-1}{2}-\alpha)$ all have solutions, the result follows from Theorem~\ref{the:layers-HWP}(i).

\item This follows immediately from (i).
\end{enumerate}
\end{proof}

Since a complete solution to the Hamilton-Waterloo Problem for $\lambda K_n$ with $\lambda \ge 2$ and both 2-factor types bipartite will be given in Section~\ref{sec:bip-Kn}, we shall limit the analogue of the above corollary for $n$ even to the case when exactly one of the 2-factor types is bipartite.

\begin{cor}\label{cor:layers-HWP-even}
Let $T_1$ and $T_2$ be admissible 2-factor types for $K_n$ with $n$ even, where $T_1$ is bipartite and $T_2$ is not. Let $\lambda \ge 2$.
\begin{enumerate}[(i)]
\item If $HWP(K_n;T_1,T_2;\frac{n-2}{2},0)$, $HWP(K_n;T_1,T_2;0,\frac{n-2}{2})$, and
    $HWP(K_n;T_1,T_2;\alpha,\frac{n-2}{2}-\alpha)$ all have solutions, then $HWP(\lambda K_n;T_1,T_2;\alpha',\lfloor \lambda \frac{n-1}{2} \rfloor -\alpha')$ has a solution for all $\alpha' \in \{ \lfloor \frac{\lambda }{2} \rfloor,\ldots,\lfloor \lambda \frac{n-1}{2} \rfloor \}$ such that $\alpha' -\lfloor \frac{\lambda }{2} \rfloor \equiv \alpha \pmod{\frac{n-2}{2}}$.
\item If $HWP(K_n;T_1,T_2;\alpha,\frac{n-2}{2}-\alpha)$ has a solution for all $\alpha \in \{ 0,\ldots,\frac{n-2}{2} \}$, then \\ $HWP(\lambda K_n;T_1,T_2;\alpha',\lfloor \lambda \frac{n-1}{2} \rfloor -\alpha')$ has a solution for all $\alpha' \in \{ \lfloor \frac{\lambda }{2} \rfloor,\ldots,\lfloor \lambda \frac{n-1}{2} \rfloor \}$.
\end{enumerate}
\end{cor}

\begin{proof}
\begin{enumerate}[(i)]
\item Take any $\alpha' \in \{ \lfloor \frac{\lambda }{2} \rfloor,\ldots,\lfloor \lambda \frac{n-1}{2} \rfloor \}$ such that $\alpha' -\lfloor \frac{\lambda }{2} \rfloor \equiv \alpha \pmod{\frac{n-2}{2}}$.
We apply Theorem~\ref{the:layers-HWP}(ii) with all $\mu_i=1$. Since $n$ is even, we'll have ${\cal O}=\{ 1,\ldots,\ell\}$ and $\beta=\lambda$, and since $T_2$ is not bipartite, we'll have $\bar{\alpha}=\lfloor \frac{\lambda }{2} \rfloor$ and $\bar{\gamma}=0$.
 It suffices to find non-negative integers $\alpha_1,\ldots,\alpha_{\ell}$ and $x_1,\ldots,x_{\ell}$ such that $\lambda=\sum_{i=1}^{\ell} x_i$, $\alpha'- \lfloor \frac{\lambda}{2} \rfloor= \sum_{i=1}^{\ell} x_i\alpha_i$, and $HWP(K_n;T_1,T_2;\alpha_i,\frac{n-2}{2}-\alpha_i)$ has a solution for all $i$.

Write $\alpha'-\lfloor \frac{\lambda }{2} \rfloor=q \frac{n-2}{2}+\alpha$. Observe that, since $\lfloor \lambda \frac{n-1}{2} \rfloor-\lfloor \frac{\lambda }{2} \rfloor=\lambda \frac{n-2}{2}$, we have $0 \le q \le \lambda$. Moreover,
if $q=\lambda$, then $\alpha=0$. In this case, we let $\ell=1$, $\alpha_1=\frac{n-2}{2}$, $x_1=q$.

Hence we may assume $q < \lambda$. Now let $\ell=3$, $\alpha_1=\frac{n-2}{2}$, $\alpha_2=\alpha$, $\alpha_3=0$, $x_1=q$, $x_2=1$, and $x_3=\lambda-q-1$.

\item This follows immediately from (i).
\end{enumerate}
\end{proof}

Many solutions to the Hamilton-Waterloo Problem for $K_n$ are known for very specific uniform 2-factor types (see \cite{MR2246267}), hence we shall limit the application of our techniques to some of the most comprehensive results.

For odd $n \le 17$ it has been shown \cite{MR2239309,MR2246267} that $HWP(K_n;T_1,T_2;\alpha,\gamma)$ has a solution for all admissible 2-factor types $T_1$ and $T_2$, and all non-negative integers $\alpha$ and $\gamma$ satisfying $\alpha+\gamma=\frac{n-1}{2}$ except in the following cases, which are known to have no solution:
\begin{description}
\item{(C7)} $n=7$, $T_1=[3,4]$, $T_2=[7]$, $\alpha=2$;
\item{(C9a)} $n=9$, $T_1=[4,5]$, $\alpha=4$;
\item{(C9b)} $n=9$, $T_1=[4,5]$, $T_2=[3^{\la 3 \ra}]$, $\alpha \in \{ 1,2 \}$;
\item{(C9c)} $n=9$, $T_1=[3^{\la 3 \ra}]$, $T_2 \in \{ [3,6],[9] \}$, $\alpha=3$;
\item{(C11)} $n=11$, $T_1=[3,3,5]$, $\alpha=5$;
\item{(C15)} $n=15$, $T_1=[3^{\la 5 \ra}]$, $T_2 \in \{ [3^{\la 2 \ra},4,5],[3,5,7],[5^{\la 3 \ra}], [4^{\la 2 \ra},7],[7,8] \}$, $\alpha=6$.
\end{description}

Our layering technique yields the following corollary.

\begin{cor}\label{cor:HWP-small}
For odd $n \le 17$ and $\lambda \ge 2$, $HWP(\lambda K_n;T_1,T_2;\alpha',\gamma')$ has a solution for all admissible 2-factor types $T_1$ and $T_2$ for $K_n$, and all non-negative integers $\alpha'$ and $\gamma'$ satisfying $\alpha'+\gamma'=\lambda \frac{n-1}{2}$ except possibly in the following cases:
\begin{description}
\item{(D7)} $n=7$, $T_1=[3,4]$, $T_2=[7]$, $\alpha' = 3\lambda-1$;
\item{(D9a)} $n=9$, $T_1=[4,5]$, $\alpha' > 3\lambda$;
\item{(D9b)} $n=9$, $T_1=[4,5]$, $T_2=[3^{\la 3 \ra}]$,  $\alpha' \not\equiv 0 \pmod{3}$;
\item{(D9c)} $n=9$, $T_1=[3^{\la 3 \ra}]$, $T_2 \in \{ [3,6],[9] \}$, $\alpha'=4\lambda-1$;
\item{(D11)} $n=11$, $T_1=[3,3,5]$, $\alpha'> 4\lambda$;
\item{(D15)} $n=15$, $T_1=[3^{\la 5 \ra}]$, $T_2 \in \{ [3^{\la 2 \ra},4,5],[3,5,7],[5^{\la 3 \ra}], [4^{\la 2 \ra},7],[7,8] \}$, $\alpha'=7\lambda-1$.
\end{description}
\end{cor}

\begin{proof}
Let $n \le 17$ be an odd integer,  $T_1$ and $T_2$ admissible 2-factor types for $K_n$, and $\alpha' \in \{ 0,\ldots,\lambda \frac{n-1}{2} \}$. Write $\alpha'=q \frac{n-1}{2}+\alpha$ for $0 \le q \le \lambda$ and $0 \le \alpha < \frac{n-1}{2}$.

{\sc Case 1:} $n \not\in \{ 9,11 \}$. Then $HWP(K_n;T_1,T_2;\frac{n-1}{2},0)$ and
$HWP(K_n;T_1,T_2;0,\frac{n-1}{2})$ both have solutions. If $n$, $T_1$, $T_2$, and $\alpha$ do not fall into Case (C7) or (C15) above, then $HWP(K_n;T_1,T_2;\alpha,\frac{n-1}{2}-\alpha)$ has a solution, and hence $HWP(\lambda K_n;T_1,T_2;\alpha',\lambda \frac{n-1}{2}-\alpha')$ has a solution by Corollary~\ref{cor:layers-HWP}(i). It remains to examine Cases (C7) and (C15). As in the proof of Corollary~\ref{cor:layers-HWP}(i), it suffices to find non-negative integers $\alpha_1,\ldots,\alpha_{\ell}$ and $x_1,\ldots,x_{\ell}$ such that $\lambda=\sum_{i=1}^{\ell} x_i$, $\alpha'= \sum_{i=1}^{\ell} x_i\alpha_i$, and $HWP(K_n;T_1,T_2;\alpha_i,\frac{n-1}{2}-\alpha_i)$ has a solution for all $i$.

(C7) Assume $n=7$, $T_1=[3,4]$, $T_2=[7]$, and $\alpha=2$. If $\alpha' \ne 3\lambda-1$, we may write $\alpha'=3q+2=q \cdot 3+2 \cdot 1 $ and take $\alpha_1=3$, $\alpha_2=1$, $\alpha_3=0$, $x_1=q$, $x_2=2$, $x_3=\lambda-(q+2)$.

(C15) Assume $n=15$, $T_1=[3^{\la 5 \ra}]$, $T_2 \in \{ [3^{\la 2 \ra},4,5],[3,5,7],[5^{\la 3 \ra}], [4^{\la 2 \ra},7],[7,8] \}$, and $\alpha=6$. If $\alpha' \ne 7\lambda-1$, we may write $\alpha'=7q+6=q \cdot 7+2 \cdot 3$ and take $\alpha_1=7$, $\alpha_2=3$, $\alpha_3=0$, $x_1=q$, $x_2=2$, $x_3=\lambda-(q+2)$.

{\sc Case 2:} $n=9$. If $T_1, T_2 \not\in \{ [4,5], [3^{\la 3 \ra}] \}$, then $HWP(K_n;T_1,T_2;\frac{n-1}{2},0)$,
$HWP(K_n;T_1,T_2;$ $0,\frac{n-1}{2})$, and
$HWP(K_n;T_1,T_2;\alpha,\frac{n-1}{2}-\alpha)$ all have solutions, and hence by Corollary~\ref{cor:layers-HWP}(i), $HWP(\lambda K_n;T_1,T_2;\alpha',\lambda\frac{n-1}{2}-\alpha')$ has a solution. It remains to consider cases with $T_1=[4,5]$ and $T_1=[3^{\la 3 \ra}]$.

Assume $T_1=[4,5]$. If $T_2 \ne [3^{\la 3 \ra}]$ and $\alpha' \le 3\lambda$, then we can write $\alpha'=\sum_{i=1}^{\lambda} \alpha_i$ with $0 \le \alpha_i \le 3$ for all $i$. We let $\ell=\lambda$ and $x_1=\ldots=x_{\ell}=1$. If $T_2 = [3^{\la 3 \ra}]$, $\alpha' \le 3\lambda$, and $\alpha' \equiv 0 \pmod{3}$, then we can write $\alpha'=q \cdot 3$ for $q \le \lambda$. We then let $\ell=2$, $\alpha_1=3$, $\alpha_2=0$, $x_1=q$, $x_2=\lambda-q$.

Assume $T_1=[3^{\la 3 \ra}]$ and, without loss of generality, $T_2 \in \{ [3,6],[9] \}$. If $\alpha \ne 3$, then as above,
$HWP(\lambda K_n;T_1,T_2;\alpha',\lambda \frac{n-1}{2}-\alpha')$ by Corollary~\ref{cor:layers-HWP}(i). Hence suppose
$\alpha=3$. If $\alpha' \ne 4\lambda-1$, we may write $\alpha'=4q+3=q \cdot 4+2+1$ and take $\alpha_1=4$, $\alpha_2=2$, $\alpha_3=1$, $\alpha_4=0$, $x_1=q$, $x_2=1$, $x_3=1$, $x_4=\lambda-(q+2)$.

{\sc Case 3:} $n=11$. If $T_1, T_2 \ne [3,3,5]$, then $HWP(K_n;T_1,T_2;\frac{n-1}{2},0)$,
$HWP(K_n;T_1,T_2;$ $0,\frac{n-1}{2})$, and
$HWP(K_n;T_1,T_2;\alpha,\frac{n-1}{2}-\alpha)$ all have solutions, and hence by Corollary~\ref{cor:layers-HWP}(i), $HWP(\lambda K_n;T_1,T_2;\alpha',\lambda\frac{n-1}{2}-\alpha')$ has a solution. It remains to consider the case $T_1=[3,3,5]$. If $\alpha' \le 4\lambda$,  we write $\alpha'=\sum_{i=1}^{\lambda} \alpha_i$ with $0 \le \alpha_i \le 4$ for all $i$. We then let $\ell=\lambda$ and $x_1=\ldots=x_{\ell}=1$.
\end{proof}

\label{HWP-even-n}

For even $n \le 10$ it has been shown \cite{MR2239309,MR2246267} that $HWP(K_n;T_1,T_2;\alpha,\gamma)$ has a solution for all admissible 2-factor types $T_1$ and $T_2$, and all non-negative integers $\alpha$ and $\gamma$ satisfying $\alpha+\gamma=\frac{n-2}{2}$ except in the following cases, which are known to have no solution:
\begin{description}
\item{(C6)} $n=6$, $T_1=[6]$, $T_2=[3,3]$, $\alpha=0$;
\item{(C8)} $n=8$, $T_1=[4^{\la 2 \ra}]$, $T_2=[3,5]$,  $\alpha \in \{ 1,2 \}$.
\end{description}

Our layering technique, however, applies only to cases where at least one of the 2-factor types is bipartite, and since the case with two bipartite 2-factor types will be fully solved in Corollary~\ref{cor:bipartite-Kn-HWP}, we shall limit the next result to the case with a single bipartite 2-factor type.

\begin{cor}\label{cor:HWP-small-even}
For even $n \le 10$ and $\lambda \ge 2$, $HWP(\lambda K_n;T_1,T_2;\alpha',\gamma')$ has a solution for all admissible 2-factor types $T_1$ and $T_2$ for $K_n$, where $T_1$ is bipartite and $T_2$ is not, and all non-negative integers $\alpha'$ and $\gamma'$ satisfying $\alpha'+\gamma'= \lfloor \lambda \frac{n-1}{2} \rfloor$ and $\alpha' \ge \lfloor  \frac{\lambda}{2} \rfloor$ except possibly in the following cases:
\begin{description}
\item{(D6)} $n=6$, $T_1=[6]$, $T_2=[3,3]$, $\alpha'<\lambda + \lfloor  \frac{\lambda}{2} \rfloor$;
\item{(D8)} $n=8$, $T_1=[4^{\la 2 \ra}]$, $T_2=[3,5]$,  $\alpha' \not\equiv \lfloor  \frac{\lambda}{2} \rfloor \pmod{3}$.
\end{description}
\end{cor}

\begin{proof}
Let $n \le 10$ be an even integer,  $T_1$ and $T_2$ admissible 2-factor types for $K_n$ with $T_1$ bipartite and $T_2$ not, and $\alpha' \in \{ \lfloor  \frac{\lambda}{2} \rfloor,\ldots,\lfloor \lambda \frac{n-1}{2} \rfloor \}$. Write $\alpha'-\lfloor  \frac{\lambda}{2} \rfloor
=q \frac{n-2}{2}+\alpha$ for $0 \le q \le \lambda$ and $0 \le \alpha < \frac{n-2}{2}$.

{\sc Case 1:} $n \ne 6$. Then $HWP(K_n;T_1,T_2;\frac{n-2}{2},0)$ and
$HWP(K_n;T_1,T_2;0,\frac{n-2}{2})$ both have solutions. If $n$, $T_1$, $T_2$, and $\alpha$ do not fall into Case (C8) above, then $HWP(K_n;T_1,T_2;\alpha,\frac{n-2}{2}-\alpha)$ has a solution, and hence $HWP(\lambda K_n;T_1,T_2;\alpha',\lfloor \lambda \frac{n-1}{2} \rfloor -\alpha')$ has a solution by Corollary~\ref{cor:layers-HWP-even}(i). It remains to examine Case (C8). Here, $n=8$, $T_1=[4^{\la 2 \ra}]$, $T_2=[3,5]$, and $\alpha \in \{ 1,2 \}$, but if $\alpha'-\lfloor  \frac{\lambda}{2} \rfloor \equiv 0 \pmod{3}$, then we have a contradiction.

{\sc Case 2:} $n=6$. Then $T_1=[6]$, $T_2=[3,3]$, and  $\alpha'-\lfloor  \frac{\lambda}{2} \rfloor = 2q + \alpha$ for $0 \le q \le \lambda$ and  $\alpha \in \{ 0,1 \}$.
As in the proof of Corollary~\ref{cor:layers-HWP-even}, it suffices to find non-negative integers $\alpha_1,\ldots,\alpha_{\ell}$ and $x_1,\ldots,x_{\ell}$ such that $\lambda=\sum_{i=1}^{\ell} x_i$, $\alpha'-\lfloor  \frac{\lambda}{2} \rfloor= \sum_{i=1}^{\ell} x_i\alpha_i$, and $HWP(K_n;T_1,T_2;\alpha_i,\frac{n-2}{2}-\alpha_i)$ has a solution for all $i$.
If $\alpha'-\lfloor  \frac{\lambda}{2} \rfloor \ge \lambda$, then we may write $\alpha'-\lfloor  \frac{\lambda}{2} \rfloor=\left( (2q+\alpha)-\lambda \right) \cdot 2 + \left(2\lambda-(2q+\alpha) \right) \cdot 1$ and take $\alpha_1=2$, $\alpha_2=1$, $x_1=(2q+\alpha)-\lambda$, $x_2=2\lambda-(2q+\alpha)$.
\end{proof}

Observe that the same results are obtained from Corollary~\ref{cor:layers-small-even-n} for $y=0$ and only two 2-factor types, of which exactly one is bipartite.

Combining results on the Oberwolfach Problem for uniform 2-factors \cite{MR1008157}
with Theorem 7.1 in \cite{MR3745157}, we can see that for odd integers $n$, $s$, $t$ such that $3 \le s \le t$, $s$ and $t$ divide $n$, and $n \ne \frac{st}{\gcd(s,t)}$,  $HWP(K_n;[s^\ast],[t^\ast];\alpha,\gamma)$ has a solution for all non-negative integers $\alpha$ and $\gamma$ such that $\alpha+\gamma=\frac{n-1}{2}$ and $\alpha \not\in \{ 1,\frac{n-3}{2},\frac{n-7}{2} \}$. Our layering technique yields the following corollary.

\begin{cor}\label{cor:HWP-odd}
Let $\lambda \ge 2$, and let $n \ge 11$ be odd. Furthermore, let $s$ and $t$ be odd integers such that $3 \le s \le t$, $s$ and $t$ divide $n$, and $n \ne \frac{st}{\gcd(s,t)}$. Then  $HWP(\lambda K_n;[s^\ast],[t^\ast];\alpha',\gamma')$ has a solution for all non-negative integers $\alpha'$ and $\gamma'$ such that $\alpha'+\gamma'=\lambda \frac{n-1}{2}$ and  $\alpha' \not\in \{1,\lambda\frac{n-1}{2}-3,\lambda\frac{n-1}{2}-1 \}$.
\end{cor}

\begin{proof}
Take any integer $\alpha'$ such that $0 \le \alpha' \le \lambda \frac{n-1}{2}$ and
$\alpha' \not\in \{1,\lambda\frac{n-1}{2}-3,\lambda\frac{n-1}{2}-1 \}$. It suffices to find non-negative integers $\alpha_1,\ldots,\alpha_{\ell}$ and $x_1,\ldots,x_{\ell}$ such that $\lambda=\sum_{i=1}^{\ell} x_i$, $\alpha'= \sum_{i=1}^{\ell} x_i\alpha_i$, and $HWP(K_n;[s^\ast],[t^\ast];\alpha_i,\frac{n-1}{2}-\alpha_i)$ has a solution for all $i$.

Write $\alpha'=q \frac{n-1}{2}+\alpha$ for $0 \le q \le \lambda$ and $0 \le \alpha < \frac{n-1}{2}$. If $q=\lambda$, then $\alpha=0$. In this case, let $\ell=1$, $\alpha_1=\frac{n-1}{2}$, $x_1=q$. Hence we may assume $q<\lambda$.

If $\alpha \not\in \{ 1, \frac{n-3}{2}, \frac{n-7}{2} \}$, let $\ell=3$, $\alpha_1=\frac{n-1}{2}$, $\alpha_2=\alpha$, $\alpha_3=0$, $x_1=q$, $x_2=1$ $x_3=\lambda-q-1$.

If $\alpha=1$, then $1 \le q < \lambda$. Write $\alpha' = q \frac{n-1}{2}+1=(q-1) \frac{n-1}{2}+ \frac{n-5}{2}+3$, and let $\ell=4$, $\alpha_1=\frac{n-1}{2}$, $\alpha_2=\frac{n-5}{2}$, $\alpha_3=3$, $\alpha_4=0$, $x_1=q-1$, $x_2=x_3=1$, $x_4=\lambda-q-1$.

If $\alpha \in \{\frac{n-3}{2},\frac{n-7}{2}\}$, then $0 \le q \le \lambda-2$. For $\alpha=\frac{n-3}{2}$, write $\alpha' = q \frac{n-1}{2}+\frac{n-3}{2}=q \frac{n-1}{2}+ \frac{n-9}{2}+3$, and let $\ell=4$, $\alpha_1=\frac{n-1}{2}$, $\alpha_2=\frac{n-9}{2}$, $\alpha_3=3$, $\alpha_4=0$, $x_1=q$, $x_2=x_3=1$, $x_4=\lambda-q-2$.
For $\alpha=\frac{n-7}{2}$, write $\alpha' = q \frac{n-1}{2}+\frac{n-7}{2}=q \frac{n-1}{2}+ \frac{n-11}{2}+2$, and let $\ell=4$, $\alpha_1=\frac{n-1}{2}$, $\alpha_2=\frac{n-11}{2}$, $\alpha_3=2$, $\alpha_4=0$, $x_1=q$, $x_2=x_3=1$, $x_4=\lambda-q-2$.
\end{proof}

Theorem 1.5 in \cite{BurDanTra-Arxiv}, together with the results on the Oberwolfach Problem with uniform 2-factors, shows that $HWP(K_n;[t^\ast],[s^\ast];\alpha,\lfloor \frac{n-1}{2} \rfloor -\alpha)$ has a solution for all $\alpha \in \{ 0,\ldots,\lfloor \frac{n-1}{2} \rfloor \}$ provided that $3 \le s < t$, $s|t$, $t|n$, $s$ is odd, $t \not\in \{2s,6s\}$, $n \not\in \{t,2t,4t\}$, $(s,n) \ne (3,6t)$, $\alpha \ne 0$ if $(s,n) \in \{ (3,6),(3,12)\}$, $\alpha \ne 1$, and $\alpha \ne 2$ if $t \equiv s \pmod{4s}$.
Our layering techniques yields the following corollary for $t$ even. Keep in mind that by \cite{MR785659,MR1008157,MR1140806}, $OP(K_n;[m^\ast])$ has a solution if and only if $m|n$ and $(m,n) \not\in \{ (3,6),(3,12) \}$.

\begin{cor}\label{cor:layers-HWP-odd+even}
Let $\lambda \ge 2$, and let $n$, $s$, $t$, and $\alpha'$ be non-negative integers such that
\begin{itemize}
\item $3 \le s < t$, $s|t$, $t|n$, $s$ is odd, and $t$ is even;
\item $t \not\in \{2s,6s\}$, $n \not\in \{t,2t,4t\}$, and $(s,n) \ne (3,6t)$;
\item $\alpha' \in \{ \lfloor \frac{\lambda }{2} \rfloor,  \ldots, \lfloor \lambda \frac{n-1}{2} \rfloor \}$, $\alpha' \ne \lfloor \frac{\lambda }{2} \rfloor+1$, and $\alpha' \ge \lambda + \lfloor \frac{\lambda }{2} \rfloor$ if $(s,n) \in \{(3,6),(3,12)\}$.
\end{itemize}
Then $HWP(\lambda K_n;[t^\ast],[s^\ast];\alpha',\lfloor \lambda \frac{n-1}{2} \rfloor -\alpha')$ has a solution.
\end{cor}

\begin{proof}
Take any $\alpha'$ satisfying the assumptions. Write $\alpha'-\lfloor \frac{\lambda}{2} \rfloor=q \frac{n-2}{2}+\alpha$, where $0 \le q \le \lambda$ and $0 \le \alpha < \frac{n-2}{2}$.

{\sc Case 1:} $\alpha \ne 1$ and $(s,n) \not\in \{(3,6),(3,12)\}$.  Then $HWP(K_n;[t^\ast],[s^\ast];0,\frac{n-2}{2})$ and $HWP(K_n;[t^\ast],[s^\ast];\frac{n-2}{2},0)$ have solutions by \cite{MR785659,MR1008157,MR1140806}, and $HWP(K_n;[t^\ast],[s^\ast];\alpha,\frac{n-2}{2} -\alpha)$ has a solution by \cite{BurDanTra-Arxiv}. Hence by Corollary~\ref{cor:layers-HWP-even}(i),  $HWP(\lambda K_n;[t^\ast],[s^\ast];\alpha',\lfloor \lambda \frac{n-1}{2} \rfloor -\alpha')$ has a solution.

{\sc Case 2:} $\alpha = 1$ and $(s,n) \not\in \{(3,6),(3,12)\}$. As in the proof of Corollary~\ref{cor:layers-HWP-even}(i),
it suffices to find non-negative integers $\alpha_1,\ldots,\alpha_{\ell}$ and $x_1,\ldots,x_{\ell}$ such that $\lambda=\sum_{i=1}^{\ell} x_i$, $\alpha'-\lfloor  \frac{\lambda}{2} \rfloor= \sum_{i=1}^{\ell} x_i\alpha_i$, and $HWP(K_n;[t^\ast],[s^\ast];\alpha_i,\frac{n-2}{2}-\alpha_i)$ has a solution for all $i$. We have $\alpha'-\lfloor \frac{\lambda}{2} \rfloor=q \frac{n-2}{2}+1=(q-1) \frac{n-2}{2}+\frac{n-4}{2}+2$. Since $\alpha' \ne \lfloor \frac{\lambda}{2} \rfloor+1$, we may then take $\alpha_1=\frac{n-2}{2}$, $\alpha_2=\frac{n-4}{2}$, $\alpha_3=2$, $\alpha_4=0$, $x_1=q-1$, $x_2=x_3=1$, and $x_4=\lambda-(q+1)$.

{\sc Case 3:} $(s,n) \in \{(3,6),(3,12)\}$. Since $\alpha' \ge \lambda + \lfloor \frac{\lambda}{2} \rfloor$, we can write
$\alpha' - \lfloor \frac{\lambda}{2} \rfloor=\sum_{i=1}^{\lambda} \alpha_i$, where $1 \le \alpha_i \le \frac{n-2}{2}$ for all $i$. Hence we may take $\ell=\lambda$ and $x_i=1$ for all $i$.
\end{proof}

A somewhat more restrictive result on $HWP(K_n;T_1,T_2;\alpha,\gamma)$ for uniform 2-factor types $T_1$ and $T_2$, with exactly one of them bipartite, is proved in \cite[Theorem 1.2]{KerPas-Arxiv}. The same authors  also have solutions to $HWP(K_n;T_1,T_2;\alpha,\gamma)$ for certain non-uniform 2-factor types \cite[Corollary 10.9]{MR3687793}, and some of these cases involve exactly one bipartite 2-factor type. Both of these results are of the following form: if the parameters satisfy some tedious arithmetic conditions, then $HWP(K_n;T_1,T_2;\alpha,\lfloor \frac{n-1}{2} \rfloor-\alpha)$ has a solution for all $\alpha \in \{ 0,\ldots, \lfloor \frac{n-1}{2} \rfloor \}$ except possibly for $\alpha \in \{ 1,\lfloor \frac{n-1}{2} \rfloor-1 \}$. We shall state the result of applying our layering technique to these solutions in a general form, without detailing the arithmetic conditions.

\begin{cor}\label{cor:KerPas}
Let $n \ge 8$ and $\lambda \ge 2$ be integers, and $T_1$, $T_2$ admissible 2-factor types for $K_n$ such that $T_1$ is bipartite and $T_2$ is not. Assume $HWP(K_n;T_1,T_2;\alpha,\frac{n-2}{2}-\alpha)$ has a solution for all $\alpha \in \{ 0,\ldots, \frac{n-2}{2}  \}$ except possibly for $\alpha \in \{ 1,\frac{n-4}{2}  \}$. Then $HWP(\lambda K_n;T_1,T_2;\alpha',\lfloor \lambda \frac{n-1}{2} \rfloor-\alpha')$ has a solution for all $\alpha' \in \{ \lfloor \frac{\lambda}{2} \rfloor,\ldots, \lfloor \lambda \frac{n-1}{2}\rfloor \}$ except possibly for $\alpha' \in \{ \lfloor \frac{\lambda}{2} \rfloor+1,\lfloor \lambda \frac{n-1}{2}\rfloor-1  \}$.
\end{cor}

\begin{proof}
By the assumptions, $HWP(K_n;T_1,T_2;0,\frac{n-2}{2})$ and
$HWP(K_n;T_1,T_2;\frac{n-2}{2},0)$ have solutions.
Take any $\alpha' \in \{ \lfloor \frac{\lambda}{2} \rfloor,\ldots, \lfloor \lambda \frac{n-1}{2}\rfloor \}$ such that $\alpha' \not\in \{ \lfloor \frac{\lambda}{2} \rfloor+1,\lfloor \lambda \frac{n-1}{2}\rfloor-1  \}$. Write $\alpha'-\lfloor \frac{\lambda}{2} \rfloor=q \frac{n-2}{2}+\alpha$, where $0 \le q \le \lambda$ and $0 \le \alpha < \frac{n-2}{2}$.

{\sc Case 1:} $\alpha \not\in \{ 1, \frac{n-4}{2} \}$. Then $HWP(K_n;T_1,T_2;\alpha,\frac{n-2}{2} -\alpha)$ has a solution by the assumption, and so by Corollary~\ref{cor:layers-HWP-even}(i), $HWP(\lambda K_n;T_1,T_2;\alpha',\lfloor \lambda \frac{n-1}{2} \rfloor -\alpha')$ has a solution.

{\sc Case 2:} $\alpha \in \{ 1, \frac{n-4}{2} \}$. As in the proof of Corollary~\ref{cor:layers-HWP-even}(i),
it suffices to find non-negative integers $\alpha_1,\ldots,\alpha_{\ell}$ and $x_1,\ldots,x_{\ell}$ such that $\lambda=\sum_{i=1}^{\ell} x_i$, $\alpha'-\lfloor \frac{\lambda}{2} \rfloor= \sum_{i=1}^{\ell} x_i\alpha_i$, and $HWP(K_n;T_1,T_2;\alpha_i,\frac{n-2}{2}-\alpha_i)$ has a solution for all $i$.

If $\alpha=1$, we write $\alpha'-\lfloor \frac{\lambda}{2} \rfloor=q \frac{n-2}{2}+1=(q-1) \frac{n-2}{2}+\frac{n-4}{2}+2$. Since $\alpha' \ne \lfloor \frac{\lambda}{2} \rfloor+1$, we may take $\alpha_1=\frac{n-2}{2}$, $\alpha_2=\frac{n-4}{2}$, $\alpha_3=2$, $\alpha_4=0$, $x_1=q-1$, $x_2=x_3=1$, and $x_4=\lambda-(q+1)$.

If $\alpha=\frac{n-4}{2}$, we write $\alpha'-\lfloor \frac{\lambda}{2} \rfloor=q \frac{n-2}{2}+\frac{n-4}{2}=q \frac{n-2}{2}+\frac{n-8}{2}+2$. Since $\alpha' \ne \lfloor \lambda \frac{n-1}{2}\rfloor-1 $, we have $\lambda \ge q+2$, and we may take $\alpha_1=\frac{n-2}{2}$, $\alpha_2=\frac{n-8}{2}$, $\alpha_3=2$, $\alpha_4=0$, $x_1=q$, $x_2=x_3=1$, and $x_4=\lambda-(q+2)$.
\end{proof}

\section{The Hamilton-Waterloo Problem for complete equi\-par\-tite multigraphs}\label{hwK_nmultip}

We shall now apply our detachment technique to the results of the previous section to obtain new solutions to the  Hamilton-Waterloo Problem for complete equipartite multigraphs.

First, Corollary~\ref{cor:direct}(iii) gives the following general result.

\begin{theo}\label{thm:HWP}
If $HWP(\lambda m K_n; T_1,T_2;\alpha_1,\alpha_2)$ has a solution, then $HWP(\lambda K_{n \times m}; mT_1,mT_2;$ $\alpha_1,\alpha_2)$ has a solution.
\end{theo}

Theorems~\ref{the:layers-HWP} and \ref{thm:HWP} then immediately imply the following.

\begin{cor}\label{cor:layers-HWP}
For $i=1,\ldots,\ell$, assume $HWP(\mu_i K_n; \bar{T_1},\bar{T_2};\alpha_i,\gamma_i)$ has a solution. Let ${\cal O}=\{ i: \mu_i(n-1) \mbox{ is odd} \}$, and  let $\lambda$ and $m$ be positive integers. Take any $(x_1,\ldots,x_{\ell}) \in \NN^{\ell}$ such that $\sum_{i=1}^{\ell} x_i \mu_i=\lambda m$, and let $\beta=\sum_{i \in {\cal O}} x_i$.
\begin{enumerate}[(i)]
\item If $\beta \le 1$, then $HWP(\lambda K_{n \times m}; m\bar{T_1},m\bar{T_2};\sum_{i=1}^{\ell} x_i\alpha_i,\sum_{i=1}^{\ell} x_i\gamma_i)$ has a solution.
\item If $T_1$ is bipartite, then $HWP(\lambda K_{n \times m}; m\bar{T_1},m\bar{T_2};\sum_{i=1}^{\ell} x_i\alpha_i+\bar{\alpha},\sum_{i=1}^{\ell} x_i\gamma_i+\bar{\gamma})$ has a solution for all nonnegative integers $\bar{\alpha}, \bar{\gamma}$ such that
\begin{itemize}
\item $\bar{\alpha}+\bar{\gamma}=\lfloor \frac{\beta}{2} \rfloor$ and
\item $\bar{\gamma}=0$ if $T_2$ is not bipartite.
\end{itemize}
\end{enumerate}
\end{cor}

Many solutions to  $HWP(K_{n \times m}; T_1,T_2;\alpha_1,\alpha_2)$ with $T_1$ and $T_2$ both of uniform types were found in \cite{MR3457814,MR3646005,MR3745157,KerPas-Arxiv}. As far as we can tell, the only results for non-uniform 2-factors are found in \cite{MR3687793}; the authors obtained extensive solutions for the complete equipartite graphs $K_{n \times m}$ with $n$ odd, assuming the parameters satisfy some tedious arithmetic conditions. We obtain the following concrete results. (Note that the Hamilton-Waterloo Problem with bipartite 2-factors is considered in Section~\ref{bip2facsecII}.)

Corollary~\ref{cor:HWP-small} and Theorem~\ref{thm:HWP} immediately yield the following.

\begin{cor}\label{cor:HWPmn-small}
For odd $n \le 17$ and $m \ge 2$, $HWP(\lambda K_{n \times m};mT_1,mT_2;\alpha',\gamma')$ has a solution for all admissible 2-factor types $T_1$ and $T_2$ for $K_n$, and all non-negative integers $\alpha'$ and $\gamma'$ satisfying $\alpha'+\gamma'=\lambda m \frac{n-1}{2}$ except possibly in the following cases:
\begin{description}
\item{(D7)} $n=7$, $T_1=[3,4]$, $T_2=[7]$, $\alpha' = 3\lambda m-1$;
\item{(D9a)} $n=9$, $T_1=[4,5]$, $\alpha' > 3\lambda m$;
\item{(D9b)} $n=9$, $T_1=[4,5]$, $T_2=[3^{\la 3 \ra}]$, $\alpha' \not\equiv 0 \pmod{3}$;
\item{(D9c)} $n=9$, $T_1=[3^{\la 3 \ra}]$, $T_2 \in \{ [3,6],[9] \}$, $\alpha'=4\lambda m-1$;
\item{(D11)} $n=11$, $T_1=[3,3,5]$, $\alpha' > 4\lambda m$;
\item{(D15)} $n=15$, $T_1=[3^{\la 5 \ra}]$, $T_2 \in \{ [3^{\la 2 \ra},4,5],[3,5,7],[5^{\la 3 \ra}], [4^{\la 2 \ra},7],[7,8] \}$, $\alpha'=7\lambda m-1$.
\end{description}
\end{cor}

Corollary~\ref{cor:HWP-small-even} and Theorem~\ref{thm:HWP} imply the following.

\begin{cor}\label{cor:HWPmn-small-even}
For even $n \le 10$ and $m \ge 2$, $HWP(\lambda K_{n \times m};mT_1,mT_2;\alpha',\gamma')$ has a solution for all admissible 2-factor types $T_1$ and $T_2$ for $K_n$, where $T_1$ is bipartite and $T_2$ is not, and all non-negative integers $\alpha'$ and $\gamma'$ satisfying $\alpha'+\gamma'= \lfloor \lambda m \frac{n-1}{2} \rfloor$ and $\alpha' \ge \lfloor  \frac{\lambda m}{2} \rfloor$ except possibly in the following cases:
\begin{description}
\item{(D6)} $n=6$, $T_1=[6]$, $T_2=[3,3]$, $\alpha'<\lambda m + \lfloor  \frac{\lambda m}{2} \rfloor$;
\item{(D8)} $n=8$, $T_1=[4^{\la 2 \ra}]$, $T_2=[3,5]$,  $\alpha' \not\equiv \lfloor  \frac{\lambda m}{2} \rfloor \pmod{3}$.
\end{description}
\end{cor}

From Corollary~\ref{cor:HWP-odd} and Theorem~\ref{thm:HWP} we obtain the following.

\begin{cor}\label{cor:HWP-odd-mn}
Let $m \ge 2$, and let $n \ge 11$ be odd. Furthermore, let $s$, $t$ be odd integers such that $3 \le s \le t$, $s$ and $t$ divide $n$, and $n \ne \frac{st}{\gcd(s,t)}$. Then  $HWP(\lambda K_{n \times m};[(ms)^\ast],[(mt)^\ast];\alpha',\gamma')$ has a solution for all non-negative integers $\alpha'$ and $\gamma'$ such that $\alpha'+\gamma'=\lambda m \frac{n-1}{2}$ and  $\alpha' \not\in \{1,\lambda m \frac{n-1}{2}-3,\lambda m\frac{n-1}{2}-1 \}$.
\end{cor}

Applying Theorem~\ref{thm:HWP} to Corollary~\ref{cor:layers-HWP-odd+even} we have the following.

\begin{cor}\label{cor:layers-HWPmn-odd+even}
Let $m \ge 2$, and $n$, $s$, $t$, and $\alpha'$ be non-negative integers such that
\begin{itemize}
\item $3 \le s < t$, $s|t$, $t|n$, $s$ is odd, and $t$ is even;
\item $t \not\in \{2s,6s\}$, $n \not\in \{t,2t,4t\}$, and $(s,n) \ne (3,6t)$;
\item $\alpha' \in \{ \lfloor \frac{\lambda m}{2} \rfloor,  \ldots, \lfloor \lambda m \frac{n-1}{2} \rfloor \}$, $\alpha' \ne \lfloor \frac{\lambda m }{2} \rfloor+1$, and $\alpha' \ge  \lambda m + \lfloor \frac{\lambda m}{2} \rfloor$ if $(s,n) \in \{(3,6),(3,12)\}$.
\end{itemize}
Then $HWP(\lambda K_{n \times m};[(mt)^\ast],[(ms)^\ast];\alpha',\lfloor \lambda m \frac{n-1}{2} \rfloor -\alpha')$ has a solution.
\end{cor}

Finally, Corollary~\ref{cor:KerPas} combined with Theorem~\ref{thm:HWP} yields the following result.

\begin{cor}\label{cor:KerPas-mn}
Let $n \ge 8$ and $m \ge 2$ be integers, and $T_1$, $T_2$ admissible 2-factor types for $K_n$ such that $T_1$ is bipartite and $T_2$ is not. If $HWP(K_n;T_1,T_2;\alpha,\frac{n-2}{2}-\alpha)$ has a solution for all $\alpha \in \{ 0,\ldots, \frac{n-2}{2}  \}$ except possibly for $\alpha \in \{ 1,\frac{n-4}{2}  \}$, then $HWP(\lambda K_{n \times m};mT_1,mT_2;\alpha',\lfloor \lambda m \frac{n-1}{2} \rfloor-\alpha')$ has a solution for all $\alpha' \in \{ \lfloor \frac{\lambda m}{2} \rfloor,\ldots, \lfloor \lambda m \frac{n-1}{2}\rfloor \}$ except possibly for $\alpha' \in \{ \lfloor \frac{\lambda m}{2} \rfloor+1,\lfloor \lambda m \frac{n-1}{2}\rfloor-1  \}$.
\end{cor}


\section{Bipartite 2-factorizations of complete multigraphs}\label{sec:bip-Kn}

Our layering technique has the greatest effect when all 2-factor types are bipartite. In this section, we first prove a general result on existence of bipartite 2-factorizations of complete multigraphs. As an immediate application,  using previous results for complete graphs, we obtain a complete solution to the Oberwolfach Problem and an almost complete solution to the Hamilton-Waterloo Problem for $\lambda K_n$ with bipartite 2-factors.

But first, we need to introduce the notion of a signature of a 2-factorization type, which will prove  a very useful tool in this context.

\begin{defn}{\rm
Let ${\cal G}$ be an $r$-regular graph and ${\cal T}=[ T_1^{\la \alpha_1 \ra},\ldots,T_k^{\la \alpha_k \ra}]$ a 2-factorization type for ${\cal G}$. Then $[\alpha_1,\ldots,\alpha_k]$ is called a {\em signature} for ${\cal T}$.
}
\end{defn}

\begin{lemma}\label{lem:signature}
Let ${\cal G}$ be an $r$-regular graph.
\begin{enumerate}[(i)]
\item Let ${\cal T}$ be a 2-factorization type for ${\cal G}$.
\begin{enumerate}[(a)]
\item If $A=[\alpha_1,\ldots,\alpha_k]$ is a signature for ${\cal T}$, then $\alpha_1+ \ldots+\alpha_k=\lfloor \frac{r}{2} \rfloor$; that is, $A$ is a refinement of $[ \lfloor \frac{r}{2} \rfloor ]$.
\item If $A$ is a signature for ${\cal T}$, then any refinement of $A$ is a signature for ${\cal T}$.
\end{enumerate}
\item Assume $A$ is a refinement of $A'$, and $A'$ is a refinement of $[ \lfloor \frac{r}{2} \rfloor ]$. If ${\cal G}$ admits a 2-factorization of type ${\cal T}$ for every 2-factorization type ${\cal T}$ satisfying property ${\cal P}$ that admits signature $A$, then ${\cal G}$ admits a 2-factorization of type ${\cal T}'$ for every 2-factorization type ${\cal T}'$ satisfying property ${\cal P}$ that admits signature $A'$.
\end{enumerate}
\end{lemma}

\begin{proof}
\begin{enumerate}[(i)]
\item Assume $A=[\alpha_1,\ldots,\alpha_k]$ is a signature for ${\cal T}$. Then ${\cal T}$ is a multiset of the form $[ T_1^{\la \alpha_1 \ra},\ldots,T_k^{\la \alpha_k \ra}]$, for some 2-factor types $T_1,\ldots,T_k$, so either $r$ or $r-1$ is equal to $2(\alpha_1+\ldots+\alpha_k)$, and (a) follows.

    To see (b), observe that
    replacing $T_1^{\la \alpha_1 \ra}$ in the  list $[ T_1^{\la \alpha_1 \ra},\ldots,T_k^{\la \alpha_k \ra}]$ with $T_1^{\la x \ra},T_1^{\la \alpha_1-x \ra}$, for any positive integer $x< \alpha_1$, yields the same multiset, so $[x,\alpha_1-x,\alpha_2,\ldots,\alpha_k]$ is also a signature for ${\cal T}$. Since any refinement of $A$ can be obtained from $A$ by a sequence of such operations, any refinement of $A$ is also a signature for ${\cal T}$.

\item Assume ${\cal G}$ admits a 2-factorization of type ${\cal T}$ for every 2-factorization type ${\cal T}$ satisfying property ${\cal P}$ that admits signature $A$. Let ${\cal T}'$ be any 2-factorization type satisfying property ${\cal P}$ that admits signature $A'$. By (b), ${\cal T}'$ also admits signature $A$, and since it also satisfies property ${\cal P}$, by assumption, ${\cal G}$ admits a 2-factorization of type ${\cal T}'$.
\end{enumerate}
\end{proof}

A classic result by H\"aggkvist \cite{MR821524} shows that if $n \equiv 2 \pmod{4}$, then $K_n$ admits a 2-factorization of type ${\cal T}$ for every admissible bipartite 2-factorization type ${\cal T}=[T_1^{\la \alpha_1 \ra},\ldots,T_k^{\la \alpha_k \ra}]$ with all $\alpha_i$ even. A much more recent paper by Bryant and Danziger \cite{MR2833961} extends this result to the case $n \equiv 0 \pmod{4}$; in this case, $K_n$ admits a 2-factorization of type ${\cal T}$ for every admissible bipartite 2-factorization type ${\cal T}=[T_1^{\la \alpha_1 \ra},\ldots,T_k^{\la \alpha_k \ra}]$ with $\alpha_1 \ge 3$ odd, and $\alpha_i$ even for all $i \ge 2$.
In our terminology, these results can then be summarized as follows.

\begin{theo}\cite{MR821524,MR2833961}\label{the:HagBryDan}
Let $n$ be a positive even integer and ${\cal T}$ an admissible bipartite 2-factorization type for $K_n$ with a signature $[\beta_1,\ldots,\beta_s]$. Suppose there exists a refinement $A$ of $[\beta_1,\ldots,\beta_s]$ such that one of the following holds:
\begin{enumerate}[(i)]
\item $n \equiv 2 \pmod{4}$ and $A=[ 2^{\la  \frac{n-2}{4} \ra }]$;
\item $n \equiv 0 \pmod{4}$  and $A=[ 3,2^{\la  \frac{n-8}{4} \ra }]$.
\end{enumerate}
Then $K_n$ admits a 2-factorization of type ${\cal T}$.
\end{theo}

Using layering, we obtain the following extension for complete multigraphs.

\begin{theo}\label{the:bipartite-Kn}
Let $n$ be a positive even integer  and ${\cal T}=[T_1^{\la \beta_1 \ra},\ldots,T_s^{\la \beta_s \ra}]$ an admissible bipartite 2-factorization type for $\lambda K_n$, with each $T_i$ an admissible bipartite 2-factor type for $K_n$. Suppose there exists a refinement $[\alpha_1,\ldots,\alpha_{k}]$ of $[\beta_1,\ldots,\beta_s]$ such that one of the following holds:
\begin{enumerate}[(i)]
\item $n \equiv 2 \pmod{4}$ and $\Big\vert \{ i: \alpha_i \mbox{ is odd} \}  \Big\vert \le \lfloor \frac{\lambda}{2} \rfloor$; or
\item $n \equiv 0 \pmod{4}$,  $\Big\vert  \{ i: \alpha_i \mbox{ is odd} \} \Big\vert  \le \lambda + \lfloor \frac{\lambda}{2} \rfloor$, and $\Big\vert  \{ i: \alpha_i \mbox{ is odd}, \alpha_i \ge 3 \} \Big\vert  \ge \lambda $.
\end{enumerate}
Then $\lambda K_n$ admits a 2-factorization of type ${\cal T}$.
\end{theo}

\begin{proof}
For $\lambda=1$, Conditions (i) and (ii) imply the assumptions of Theorem~\ref{the:HagBryDan}, so the result is immediate. Hence assume $\lambda \ge 2$.

By Lemma~\ref{lem:signature}(ii), we may assume that ${\cal T}=[ T_1^{\la \alpha_1 \ra},\ldots,T_k^{\la \alpha_k \ra}]$ and that its signature $[\alpha_1,\ldots,\alpha_k]$ satisfies (i) or (ii).
Keep in mind that $\sum_{i=1}^k \alpha_i=\lfloor \frac{\lambda (n-1)}{2} \rfloor=\lambda \frac{n-2}{2}+ \lfloor \frac{\lambda}{2} \rfloor$.

\begin{enumerate}[(i)]
\item Let $n \equiv 2 \pmod{4}$, so $\frac{n-2}{2}$ is even. Since $\Big\vert \{ i: \alpha_i \mbox{ is odd} \} \Big\vert \le \lfloor \frac{\lambda}{2} \rfloor$, we may assume (taking a further refinement of $[\alpha_1,\ldots,\alpha_k]$ if necessary) that for some $t \le k$ we have that
$\alpha_1=\ldots=\alpha_{t}=2$  and $\sum_{i=1}^{t} \alpha_i=\lambda \frac{n-2}{2}$.
Let $\{ P_1,\ldots,P_{\lambda} \}$ be a partition of $\{ 1,\ldots,t \}$ such that $[ \alpha_i: i \in P_j ]=[ 2^{\la \frac{n-2}{4}\ra} ]$ for all $j=1,\ldots,\lambda$. By Theorem~\ref{the:HagBryDan}, $K_n$ admits a 2-factorization of type $\T_j=[ T_i^{\la \alpha_i \ra}: i \in P_j ]$ for all $j=1,\ldots,\lambda$. Let $T_1',\ldots,T_{\lfloor \frac{\lambda}{2} \rfloor}'$ be 2-factor types such that
$${\cal T} = [T_1^{\la \alpha_1 \ra},\ldots,T_k^{\la \alpha_k \ra}] = [T_1^{\la \alpha_1 \ra},\ldots,T_{t}^{\la \alpha_{t} \ra}] \sqcup [T_1',\ldots,T_{\lfloor \frac{\lambda}{2} \rfloor}']=
    \sqcup_{j=1}^{\lambda} \T_j \sqcup [T_1',\ldots,T_{\lfloor \frac{\lambda}{2} \rfloor}'].$$
We now use Theorem~\ref{the:layers} with $\ell=\lambda$, $\mu_i=x_i=1$ for all $i=1,\ldots,\lambda$, ${\cal O}=\{ 1,\ldots,\lambda \}$, and $\beta=\lambda$ to obtain a 2-factorization of $\lambda K_n$ of type $\T$.

\item Let $n \equiv 0 \pmod{4}$, so $\frac{n-2}{2}$ is odd, and observe that $n \ge 8$ since for $n=4$, ${\cal T}$ has signature $[1]$.
Given the assumptions on the number of indices $i$ such that $\alpha_i$ is odd, and replacing $[\alpha_1,\ldots,\alpha_k]$ with an appropriate refinement if necessary, we may assume that $\Big\vert  \{ i: \alpha_i \mbox{ is odd}, \alpha_i \ge 3 \} \Big\vert  = \Big\vert  \{ i: \alpha_i=3 \} \Big\vert =
\lambda $ and
$\Big\vert  \{ i: \alpha_i=1 \} \Big\vert  \le \lfloor \frac{\lambda}{2} \rfloor$. Moreover, since $\sum_{i=1}^k \alpha_i=\lambda \frac{n-2}{2}+ \lfloor \frac{\lambda}{2} \rfloor$ and $\frac{n-2}{2}$ is odd and at least 3, we may further assume that there exists $t$, $\lambda \le t \le k$, such that $\alpha_1=\ldots=\alpha_{\lambda}=3$, $\alpha_{\lambda+1}= \ldots = \alpha_{t}=2$, and
$\sum_{i=1}^{t} \alpha_i=\lambda \frac{n-2}{2}$. Then there is a partition $\{ P_1,\ldots,P_{\lambda} \}$ of $\{ 1,\ldots, t\}$ such that $[\alpha_i: i \in P_j] = [3,2^{\la \frac{n-8}{4} \ra} ]$ for all $j=1,\ldots,\lambda$.
 Hence by Theorem~\ref{the:HagBryDan}, for each $j=1,\ldots,\lambda$, there exists a 2-factorization of $K_n$ of type $[ T_i^{\la \alpha_i \ra}: i \in P_j]$. The proof is then completed using Theorem~\ref{the:layers} just as in Case (i).
\end{enumerate}
\end{proof}

Note that in the statement of Theorem~\ref{the:bipartite-Kn} we assumed that no 2-factor contains 2-cycles. However, the proof shows that the result can be extended as follows.

\begin{cor}\label{cor:bipartite-Kn}
Let $n$ be a positive even integer  and ${\cal T}=[T_1^{\la \beta_1 \ra},\ldots,T_s^{\la \beta_s \ra}]$ an admissible bipartite 2-factorization type for $\lambda K_n$. Suppose there exists a refinement $[\alpha_1,\ldots,\alpha_{k}]$ of $[\beta_1,\ldots,\beta_s]$ such that one of the following holds:
\begin{enumerate}[(i)]
\item $n \equiv 2 \pmod{4}$ and $\displaystyle{\sum_{i \atop {2 \in T_i}} \alpha_i + \Big\vert \{ i: \alpha_i \mbox{ is odd}, 2 \not\in T_i \} \Big\vert \le \Big\lfloor \frac{\lambda}{2} \Big\rfloor}$; or
\item $n \equiv 0 \pmod{4}$,  $\displaystyle{\sum_{i \atop {2 \in T_i}} \alpha_i + \Big\vert  \{ i: \alpha_i \mbox{ is odd},2 \not\in T_i \} \Big\vert  \le \lambda + \Big\lfloor \frac{\lambda}{2} \Big\rfloor}$, and $\Big\vert  \{ i: \alpha_i \mbox{ is odd}, \alpha_i \ge 3 \} \Big\vert  \ge \lambda $.
\end{enumerate}
Then $\lambda K_n$ admits a 2-factorization of type ${\cal T}$.
\end{cor}

Theorem~\ref{the:bipartite-Kn} yields the following new results on the Oberwolfach Problem and the Hamilton-Waterloo problem for $\lambda K_n$. Note that the cases with $\lambda=1$ of Corollaries~\ref{cor:bipartite-Kn-OP} and \ref{cor:bipartite-Kn-HWP} have already been solved  in \cite{MR2833961}, but we shall include them here for completeness.

\begin{cor}\label{cor:bipartite-Kn-OP}
Let $n$ be even, and $T$ any admissible bipartite 2-factor type for $K_n$. Then $OP(\lambda K_n;T)$ has a solution.
\end{cor}

\begin{proof}
If $n=4$, then $T$ is uniform, so the result follows from \cite{MR1468840}.
For $n \ge 6$, we use Theorem~\ref{the:bipartite-Kn} with  ${\cal T}=[T^{\la \alpha \ra}]$, for $\alpha=\lfloor \frac{\lambda (n-1)}{2} \rfloor$.

First assume $\lambda=1$, so $\alpha=\frac{n-2}{2}$. If $n \equiv 2 \pmod{4}$, then $\alpha$ is even, and Condition (i) in Theorem~\ref{the:HagBryDan} holds, and if $n \equiv 0 \pmod{4}$, then $\alpha$ is odd and $\alpha \ge 3$, so Condition (ii) in Theorem~\ref{the:HagBryDan} holds. In both cases, the result follows.

Hence assume $\lambda \ge 2$. It suffices to find a refinement $[\alpha_1,\ldots,\alpha_k]$ of $[\alpha]$ satisfying Condition (i) or (ii) of Theorem~\ref{the:bipartite-Kn}. If $n \equiv 2 \pmod{4}$, then $[\alpha]$ has a refinement $[2^{\la \frac{\alpha}{2} \ra}]$ or $[2^{\la \frac{\alpha-1}{2}\ra},1]$, so (i) clearly holds. For $n \equiv 0 \pmod{4}$, first observe that  $\alpha \ge \frac{1}{2}(\lambda (n-1)-1) \ge 3 \lambda+1$ if $n \ge 12$ or $\lambda \ge 3$, and  $\alpha=7=3\lambda+1$ if $n=8$ and $\lambda=2$. Hence we can write $\alpha=3\lambda + 2x$ or $\alpha=3\lambda + 2x+1$ for some non-negative integer $x$, and one of $[3^{\la \lambda \ra},2^{\la x \ra},1]$, $[3^{\la \lambda \ra},2^{\la x \ra}]$, and $[3^{\la \lambda \ra},1]$ is the required refinement of $[\alpha]$.
The result follows by Theorem~\ref{the:bipartite-Kn}.
\end{proof}

\begin{cor}\label{cor:bipartite-Kn-HWP}
Let $T_1,T_2$ be any distinct admissible bipartite 2-factor types for $K_n$. Then $HWP(\lambda K_n;T_1,T_2;\beta_1,\beta_2)$  has a solution for all positive integers $\beta_1,\beta_2$ such that
\begin{itemize}
\item $\beta_1+\beta_2=\lfloor \frac{\lambda (n-1)}{2} \rfloor$; and
\item if $\lambda=1$, then at most one of $\beta_1,\beta_2$ is odd and $1 \not\in \{ \beta_1,\beta_2\}$.
\end{itemize}
\end{cor}

\begin{proof}
For $n \in \{ 4,6\}$, there is a unique bipartite 2-factor type for $K_n$, so we must have $n \ge 8$. Let $\beta_1,\beta_2$ be any positive integers satisfying the conditions of the corollary.
We shall use Theorem~\ref{the:bipartite-Kn} with  ${\cal T}=[T_1^{\la \beta_1 \ra}, T_2^{\la \beta_2 \ra}]$.

First assume $\lambda=1$, so $\beta_1+\beta_2=\frac{n-2}{2}$. By the assumption, if $n \equiv 2 \pmod{4}$, then  $\beta_1$ and $\beta_2$ are both even,  and if $n \equiv 0 \pmod{4}$, then exactly one of $\beta_1$ and $\beta_2$ is odd, and neither is equal to 1. It follows that $[\beta_1,\beta_2]$ has a refinement satisfying one of (i) and (ii) in Theorem~\ref{the:HagBryDan}, and $K_n$ possesses a 2-factorization of type $[T_1^{\la \beta_1 \ra}, T_2^{\la \beta_2 \ra}]$.

Hence assume $\lambda \ge 2$.
Suppose first that $n \equiv 2 \pmod{4}$. We have $\Big\vert \{ i: \beta_i \mbox{ is odd} \} \Big\vert \le 2$, so condition (i) of Theorem~\ref{the:bipartite-Kn} holds whenever $\lambda \ge 4$. If $\lambda \in \{2,3\}$, then $\beta_1+\beta_2$ is odd, so $\Big\vert \{ i: \beta_i \mbox{ is odd} \} \Big\vert = 1 = \lfloor \frac{\lambda}{2} \rfloor$; again, Condition (i) holds and the result follows from Theorem~\ref{the:bipartite-Kn}.
It remains to examine the case $n \equiv 0 \pmod{4}$.

{\sc Case 1:} $n \ne 8$ or $(\lambda,\beta_1, \beta_2) \not\in \{ (2,2,5),(3,2,8),(3,5,5) \}$. We seek to find a refinement of $[\beta_1,\beta_2]$ that satisfies Condition (ii) of Theorem~\ref{the:bipartite-Kn}.
To that end, we show that we can write
\vspace{-5mm}
\begin{eqnarray}
\beta_1 =& 3x+2y_1 +\delta_1  \quad \mbox{ and} \nonumber \\
\beta_2 =& 3(\lambda-x)+2y_2 +\delta_2   \label{eq:split}
\end{eqnarray}
for some non-negative integers $x,y_1,y_2$, and $\delta_1,\delta_2 \in \{ 0,1 \}$.
This is clearly possible if $\beta_1+\beta_2 \ge 3\lambda+2$, however, it may not possible if $\beta_1+\beta_2 \le 3\lambda+1$ (for example, when $\beta_1=3x+2$ and $\beta_2=3(\lambda -x -1)+2$).

If $\lambda$ is even, then $\beta_1+\beta_2=\frac{\lambda}{2}(n-1)$, and $\frac{\lambda}{2}(n-1) \ge 3\lambda +2$ is equivalent to $\lambda(n-7) \ge 4$. This clearly holds if $\lambda \ge 4$ or $n \ge 12$. If $\lambda$ is odd, then $\beta_1+\beta_2=\frac{1}{2}(\lambda (n-1)-1)$ and $\frac{1}{2}(\lambda (n-1)-1) \ge 3\lambda +2$ is equivalent to $\lambda(n-7) \ge 5$, which holds if $\lambda \ge 5$ or $n \ge 12$. The remaining case is $n=8$ and $\lambda \in \{ 2,3\}$. In this case $\beta_1+\beta_2=3\lambda+1$, and (\ref{eq:split}) has no desired solutions if and only if $\beta_1, \beta_2 \equiv 2 \pmod{3}$, that is, if and only if $(\lambda,\beta_1, \beta_2)$ is one of the excluded cases.

We conclude that, with our assumption, System (\ref{eq:split}) has a desired solution, and hence $[3^{\la \lambda \ra},2^{\la y_1+y_2\ra},1^{\la \delta_1+\delta_2\ra}]$ (where $y_1+y_2=0$ or $\delta_1+\delta_2=0$ may be the case) is a refinement of $[\beta_1,\beta_2]$. This refinement satisfies Condition (ii) of Theorem~\ref{the:bipartite-Kn} if $\lambda \ge 4$ or $\delta_1+\delta_2\le 1$. Suppose, to the contrary, that $\lambda \in \{ 2,3\}$ with $\delta_1+\delta_2=2$. We thus have $\beta_1+\beta_2=3\lambda +2(y_1+y_2)+2$, so that $\beta_1+\beta_2 \equiv \lambda \pmod{2}$. However, for $\lambda=2$, we have that $\beta_1+\beta_2=n-1$, which is odd, and for $\lambda=3$, we have that $\beta_1+\beta_2=\frac{1}{2}(3(n-1)-1)$, which is even --- a contradiction. Hence $[ \beta_1 ,\beta_2 ]$ has a refinement satisfying Condition (ii) in Theorem~\ref{the:bipartite-Kn}, and so $\lambda K_n$ admits a 2-factorization of type $[T_1^{\la \beta_1 \ra}, T_2^{\la \beta_2 \ra}]$.

{\sc Case 2:} $n=8$ and $(\lambda,\beta_1, \beta_2) \in \{ (2,2,5),(3,2,8),(3,5,5) \}$.  Recall from page~\pageref{HWP-even-n} that $HWP(K_8;T_1,T_2;\alpha,\frac{n-2}{2}-\alpha)$ has a solution for all 2-factor types $T_1$ and $T_2$ and all $\alpha$, except in Case (C8); in particular, it has a solution  for all bipartite 2-factor types $T_1$ and $T_2$. We now use Theorem~\ref{the:layers-HWP}(ii) with  all $\mu_i=1$. We have ${\cal O}=\{ 1,\ldots,\ell\}$ and  $\beta=\lambda$.
It suffices to find non-negative integers $\alpha_1,\ldots,\alpha_{\ell}$ and $x_1,\ldots,x_{\ell}$, as well as $\bar{\alpha} \in \{ 0,\ldots, \lfloor \frac{\lambda }{2} \rfloor \}$, such that $\lambda=\sum_{i=1}^{\ell} x_i$, $\beta_1= \sum_{i=1}^{\ell} x_i\alpha_i+\bar{\alpha}$, and $HWP(K_n;T_1,T_2;\alpha_i,\frac{n-2}{2}-\alpha_i)$ has a solution for all $i$. Such suitable integers are found as follows:
\begin{itemize}
\item for $\beta_1=2$, we take $\ell=2$, $\alpha_1=2$, $\alpha_2=0$, $\bar{\alpha}=0$ $x_1=1$, $x_2=\lambda-1$, and
\item for $\beta_1=5$, we take $\ell=3$, $\alpha_1=3$, $\alpha_2=2$, $\alpha_3=0$, $\bar{\alpha}=0$ $x_1=x_2=1$, $x_3=\lambda-2$.
\end{itemize}
It then follows from Theorem~\ref{the:layers-HWP}(ii) that $HWP(\lambda K_8;T_1,T_2;\beta_1,\beta_2)$ has a solution.
\end{proof}

\section{Bipartite 2-factorizations of complete equipartite multigraphs}\label{bip2facsecII}

We shall now use detachment to extend the results of the previous section to $\lambda K_{n \times m}$ with 2-factors whose cycles have lengths that are even multiples of $m$.

\begin{theo}\label{the:bipartite-Knm}
Let $n$ be a positive even integer  and ${\cal T}=[T_1^{\la \beta_1 \ra},\ldots,T_s^{\la \beta_s \ra}]$ an admissible bipartite 2-factorization type for $\lambda mK_n$, with each $T_i$ an admissible bipartite 2-factor type for $K_n$. Suppose there exists a refinement $[\alpha_1,\ldots,\alpha_{k}]$ of $[\beta_1,\ldots,\beta_s]$ such that one of the following holds:
\begin{enumerate}[(i)]
\item $n \equiv 2 \pmod{4}$ and $\Big\vert \{ i: \alpha_i \mbox{ is odd} \}  \Big\vert \le \lfloor \frac{\lambda m}{2} \rfloor$;
\item $n \equiv 0 \pmod{4}$,  $\Big\vert  \{ i: \alpha_i \mbox{ is odd} \} \Big\vert  \le \lambda m + \lfloor \frac{\lambda m}{2} \rfloor$, and $\Big\vert  \{ i: \alpha_i \mbox{ is odd}, \alpha_i \ge 3 \} \Big\vert  \ge \lambda m $.
\end{enumerate}
Then $\lambda K_{n \times m}$ admits a 2-factorization of type $[(mT_1)^{\la \beta_1 \ra},\ldots,$ $(mT_s)^{\la \beta_s \ra}]$.
\end{theo}

\begin{proof}
With the given assumptions, Theorem~\ref{the:bipartite-Kn} guarantees existence of a 2-factorization of $\lambda mK_n$ of type ${\cal T}=[T_1^{\la \beta_1 \ra},\ldots,T_s^{\la \beta_s \ra}]$. Hence by Corollary~\ref{cor:direct}(iii), $\lambda K_{n \times m}$ admits a 2-factorization of type $[(mT_1)^{\la \beta_1 \ra},\ldots,(mT_s)^{\la \beta_s \ra}]$.
\end{proof}

The following new result on the Oberwolfach Problem for complete equipartite multigraphs follows easily from our earlier observations.

\begin{cor}\label{cor:bipartite-Knm-OP}
Let $n$ be even, and $T$ any admissible bipartite 2-factor type for $K_n$. Then $OP(\lambda K_{n \times m};mT)$ has a solution.
\end{cor}

\begin{proof}
By Corollary~\ref{cor:bipartite-Kn-OP}, $OP(\lambda m K_n;T)$ has a solution, and hence by Theorem~\ref{thm:OP}, $OP(\lambda K_{n \times m};mT)$ has a solution.
\end{proof}

For $m$ even, a more general result is obtained as a corollary to a theorem by Bryant and Danziger \cite{MR3312138}.

\begin{cor}\label{cor:bipartite-Knm-OP2}
Let $m$ be even, and $T$ any admissible bipartite 2-factor type for $K_{n \times m}$ such that $(m,n,T) \ne (6,2,[6,6])$. Then $OP(\lambda K_{n \times m};T)$ has a solution.
\end{cor}

\begin{proof}
By \cite[Theorem 12]{MR3312138}, $OP(K_{n \times m};T)$ with $m$ even has a solution for every admissible bipartite 2-factor type $T$, except that $OP(K_{2 \times 6};[6,6])$ has no solution. With our assumptions, since $K_{n \times m}$ is of even degree, it follows from Lemma~\ref{lem:basic-layers} that $OP(\lambda K_{n \times m};T)$ has a solution.
\end{proof}

As far as we know, no results are known for the Hamilton-Waterloo Problem for the complete equipartite (multi)graph with  bipartite 2-factors. We obtain the following.

\begin{cor}\label{cor:bipartite-Knm-HWP}
Let $T_1,T_2$ be any distinct admissible bipartite 2-factor types for $K_n$. Then $HWP(\lambda K_{n \times m};mT_1,mT_2;\beta_1,\beta_2)$  has a solution for all positive integers $\beta_1,\beta_2$ such that
\begin{itemize}
\item $\beta_1+\beta_2=\lfloor \frac{\lambda m(n-1)}{2} \rfloor$; and
\item if $\lambda m=1$, then at most one of $\beta_1,\beta_2$ is odd and $1 \not\in \{ \beta_1,\beta_2\}$.
\end{itemize}
\end{cor}

\begin{proof}
Let $\beta_1,\beta_2$ be any positive integers satisfying the conditions of the corollary. Then $HWP(\lambda m K_{n};T_1,T_2;\beta_1,\beta_2)$ has a solution by Corollary~\ref{cor:bipartite-Kn-HWP}, and hence $HWP(\lambda K_{n \times m};mT_1,$ $mT_2;\beta_1,\beta_2)$  has a solution by Theorem~\ref{thm:HWP}.
\end{proof}

\section{2-factorizations of $\lambda K_n$ and $\lambda K_{n \times m}$ for small even $n$}\label{sec:new-gen}

For $n$ odd, many 2-factorizations of $\lambda K_n$ of specified types can be obtained by applying Lemma~\ref{lem:basic-layers}  to the 2-factorizations of $K_n$ guaranteed to exist by \cite[Theorem 12.31]{MR2246267}. However, these results are rather obvious, so we shall not state them here explicitly. Less obvious are  2-factorizations of $\lambda K_n$, for $n$ even, obtained by applying Theorem~\ref{the:layers} to \cite[Theorem 12.33]{MR2246267}. In this section, we explicitly state and prove the results obtained for small even $n$.

\begin{cor}\label{cor:layers-small-even-n}
Let $n \in \{4,6,8,10\}$, and let $\lambda$ and $y$ be integers with $\lambda \ge 1$ and $0 \le y \le \lfloor \frac{\lambda}{2} \rfloor$. Furthermore, let $T'_1,\ldots,T'_{y}$ be any bipartite 2-factor types that are admissible for $2K_n$ but not for $K_n$.
\begin{enumerate}[(i)]
\item If $n=4$, then $\lambda K_n$ admits a 2-factorization of type
$\Big[ [4]^{\la a \ra},T'_1,\ldots,T'_{y} \Big]$ for all non-negative integers $a$ such that $a+y=\lfloor \frac{\lambda(n-1)}{2} \rfloor$.
\item If $n=6$, then $\lambda K_n$ admits a 2-factorization of type
$\Big[ [6]^{\la a \ra}, [3,3]^{\la b \ra}, T'_1,\ldots,T'_{y}\Big]$ for all  non-negative integers $a,b$ such that
    \begin{itemize}
    \item $a+b+y=\lfloor \frac{\lambda(n-1)}{2} \rfloor$, and
    \item $b \le \lambda$.
    \end{itemize}
\item If $n=8$, then $\lambda K_n$ admits a 2-factorization of type
$\Big[ [8]^{\la a \ra}, [5,3]^{\la b \ra},[4,4]^{\la c \ra}, T'_1,\ldots,T'_{y} \Big]$ for all  non-negative integers $a,b,c$ such that
    \begin{itemize}
    \item $a+b+c+y=\lfloor \frac{\lambda(n-1)}{2} \rfloor$,
    \item $a+c+y \ge \lfloor \frac{\lambda}{2} \rfloor$, and
    \item if $a=0$, then $c+y \equiv \lfloor \frac{\lambda}{2} \rfloor \pmod{3}$.
    \end{itemize}
 \item If $n=10$, then $\lambda K_n$ admits a 2-factorization of type
$\Big[ [10]^{\la a \ra}, [7,3]^{\la b \ra},[6,4]^{\la c \ra},[5,5]^{\la d \ra},$ $[4,3,3]^{\la e \ra},$ $ T'_1,\ldots,T'_{y} \Big]$ for all  non-negative integers $a,b,c,d,e$ such that
    \begin{itemize}
    \item $a+b+c+d+e+y=\lfloor \frac{\lambda(n-1)}{2} \rfloor$ and
    \item $a+c+y \ge \lfloor \frac{\lambda}{2} \rfloor$.
    \end{itemize}
\end{enumerate}
\end{cor}

\begin{proof}
By Theorem~\ref{the:layers}, if $K_n$ admits a 2-factorization of type ${\cal T}_i$ for all $i=1,\ldots,\lambda$, then $\lambda K_n$ admits a 2-factorization of type $\sqcup_{i=1}^{\lambda} {\cal T}_i \sqcup [T_1,\ldots,T_{\lfloor \frac{\lambda}{2} \rfloor}]$ for all admissible bipartite 2-factor types $T_1,\ldots,T_{\lfloor \frac{\lambda}{2} \rfloor}$ for $2K_n$. We shall apply this simplified statement to the results of \cite[Theorem 12.33(1)]{MR2246267}. In all cases we assume that parameters $y,a,b,  \ldots$ satisfy the stated conditions. In each case, we
let $T_i=T'_i$ for $i=1,\ldots,y$, and we find appropriate  $T_{y+1},\ldots,T_{\lfloor \frac{\lambda}{2} \rfloor}$ and $\T_1,\ldots,\T_{\lambda}$ below so that the resulting 2-factorization of $\lambda K_n$ is of the required type.

\begin{enumerate}[(i)]
\item Let $n=4$. Clearly, $K_4$ admits a 2-factorization of type $\Big[ [4] \Big]$. Hence we can take $T_i=[4]$ for $i=y+1,\ldots,\lfloor \frac{\lambda}{2} \rfloor$, and ${\cal T}_i=\Big[ [4] \Big]$ for all $i=1,\ldots,\lambda$. Since by assumption, $\lambda + \lfloor \frac{\lambda}{2} \rfloor -y =a$, we obtain a 2-factorization of the required type.

\item Let $n=6$. By \cite[Theorem 12.33(1)]{MR2246267}, $K_6$ admits a 2-factorization of type $\Big[ [6]^{\la s \ra}, [3,3]^{\la t \ra} \Big]$ for all non-negative integers $s,t$ with $s+t=2$ with the definite exception of $(s,t)=(0,2)$. We now apply the simplified statement of Theorem~\ref{the:layers} as stated above. Since $b \le \lambda$, we have $a+y \ge \lfloor \frac{\lambda}{2} \rfloor$, so we can take $T_i=[6]$ for $i=y+1,\ldots,\lfloor \frac{\lambda}{2} \rfloor$. Moreover, we can choose ${\cal T}_i=\Big[ [6],[3,3] \Big]$ for $i=1,\ldots,b$, and ${\cal T}_i=\Big[ [6],[6] \Big]$ for $i=b+1,\ldots,\lambda$. It can be easily verified that the resulting 2-factorization for $\lambda K_6$ contains exactly $(\lfloor \frac{\lambda}{2} \rfloor-y)+b+2(\lambda-b)=
    \lfloor \frac{\lambda (n-1)}{2} \rfloor -b-y=a$ 2-factors of type $[6]$.

\item Let $n=8$. By \cite[Theorem 12.33(1)]{MR2246267}, $K_8$ admits a 2-factorization of type $\Big[ [8]^{\la s \ra}, [5,3]^{\la t \ra},$ $[4,4]^{\la r \ra}  \Big]$ for all non-negative integers $s,t,r$ with $s+t+r=3$ with the definite exceptions of $(s,t,r)=(0,1,2)$ and $(0,2,1)$.
    Let $c'=\min\{ c, \lfloor \frac{\lambda}{2} \rfloor - y \}$ and $a'=\lfloor \frac{\lambda}{2} \rfloor - (y+c')$.
    We take $T_i = [4,4]$ for $i=y+1,\ldots,y+c'$, and $T_i=[8]$ for $i=y+c'+1,\ldots,\lfloor \frac{\lambda}{2} \rfloor$. We now have $a-a'$ copies of the 2-factor type $[8]$, $b$ copies of the 2-factor type $[5,3]$, and $c-c'$ copies of the 2-factor type $[4,4]$ to distribute among the $\lambda$ copies of $K_{8}-I$. If $b=0$, this can be done arbitrarily. Otherwise, let $c-c'=3k+\delta$ for $\delta \in \{ 0,1,2 \}$. Take ${\cal T}_i=\Big[[4,4]^{\la 3 \ra} \Big]$ for $i=1,\ldots,k$. If $\delta=0$, then the remaining 2-factor types can be arbitrarily distributed among the $\lambda-k$ copies of $K_{8}-I$. Hence suppose $\delta \ne 0$. Then $c-c' \ne 0$, so $c'=\lfloor \frac{\lambda}{2} \rfloor - y$ and $a'=0$. If $a=0$, then by the assumption $c-c'=c+y-\lfloor \frac{\lambda}{2} \rfloor \equiv 0 \pmod{3}$, so $\delta=0$, a contradiction. Hence $a \ge 1$, and we may take ${\cal T}_{k+1}=\Big[ [8],[5,3]^{\la 2-\delta \ra},[4,4]^{\la \delta \ra} \Big]$. Finally, we distribute the remaining $a-1$ copies of 2-factor type $[8]$ and $b-(2-\delta)$ copies of 2-factor type $[5,3]$ arbitrarily among the remaining $\lambda-(k+1)$ copies of $K_{8}-I$.

\item Let $n=10$. By \cite[Theorem 12.33(1)]{MR2246267}, $K_{10}$ admits a 2-factorization of type $\Big[ [10]^{\la s \ra}, [7,3]^{\la t \ra},$ $[6,4]^{\la r \ra},[5,5]^{\la q \ra},[4,3,3]^{\la u \ra} \Big]$ for all non-negative integers $s,t,r,q,u$ with $s+t+r+q+u=4$. We thus take  $T_i \in \{[10],[6,4] \}$ for $i=y+1,\ldots,\lfloor \frac{\lambda}{2} \rfloor$. The remaining 2-factor types can be arbitrarily distributed among the $\lambda$ copies of $K_{10}-I$.
\end{enumerate}
\end{proof}

Analogous results for $\lambda K_{n \times m}$ with small even $n$ (Corollaries~\ref{cor:layers-small-even-nm} and \ref{cor:layers-small-even-nm-2} below) are obtained by combining Corollary~\ref{cor:layers-small-even-n} with Corollary~\ref{cor:layers} and Theorem~\ref{the:layers2}, respectively. The latter requires more work, but for $m$ odd, it allows us to obtain some 2-factors with cycles of length not a multiple of $m$.

\begin{cor}\label{cor:layers-small-even-nm}
Let $n \in \{4,6,8,10\}$, and let $\mu$, $m$,  and $y$ be integers with $\mu \ge 1$, $m \ge 2$, and $0 \le y \le \lfloor \frac{\mu m}{2} \rfloor$. Furthermore, let $T'_1,\ldots,T'_{y}$ be any bipartite 2-factor types that are admissible for $2K_n$ but not for $K_n$.
\begin{enumerate}[(i)]
\item If $n=4$, then $\mu K_{n \times m}$ admits a 2-factorization of type
$\Big[ [4m]^{\la a \ra}, mT'_1,\ldots,mT'_{y} \Big]$ for all non-negative integers $a$ such that $a+y=\lfloor \frac{\mu m(n-1)}{2} \rfloor$.

\item If $n=6$, then $\mu K_{n \times m}$ admits a 2-factorization of type
$\Big[ [6m]^{\la a \ra}, [3m,3m]^{\la b \ra}, mT'_1,\ldots,mT'_{y} \Big]$ for all  non-negative integers $a,b$ such that
    \begin{itemize}
    \item $a+b+y=\lfloor \frac{\mu m(n-1)}{2} \rfloor$, and
    \item $b \le \mu m$.
    \end{itemize}
\item If $n=8$, then $\mu K_{n \times m}$ admits a 2-factorization of type
$\Big[ [8m]^{\la a \ra}, [5m,3m]^{\la b \ra},[4m,4m]^{\la c \ra},$ $ mT'_1,\ldots,mT'_{y} \Big]$ for all  non-negative integers $a,b,c$ such that
    \begin{itemize}
    \item $a+b+c+y=\lfloor \frac{\mu m(n-1)}{2} \rfloor$,
    \item $a+c+y \ge \lfloor \frac{\mu m}{2} \rfloor$, and
    \item if $a=0$, then $c+y \equiv \lfloor \frac{\mu m}{2} \rfloor \pmod{3}$.
    \end{itemize}
\item If $n=10$, then $\mu K_{n \times m}$ admits a 2-factorization of type
$\Big[ [10m]^{\la a \ra}, [7m,3m]^{\la b \ra},[6m,4m]^{\la c \ra},$ $[5m,5m]^{\la d \ra},[4m,3m,3m]^{\la e \ra}, mT'_1,\ldots,mT'_{y} \Big]$ for all  non-negative integers $a,b,c,d,e$ such that
    \begin{itemize}
    \item $a+b+c+d+e+y=\lfloor \frac{\mu m(n-1)}{2} \rfloor$ and
    \item $a+c+y \ge \lfloor \frac{\mu m}{2} \rfloor$.
    \end{itemize}
\end{enumerate}
\end{cor}

\begin{proof}
The results are obtained by first applying Corollary~\ref{cor:layers-small-even-n} with $\lambda =\mu m$, and then Corollary~\ref{cor:direct}(iii).
\end{proof}


\begin{cor}\label{cor:layers-small-even-nm-2}
Let $n \in \{4,6,8,10\}$, let $\beta$ and $\mu$ be  integers such that $1 \le \beta \le \mu$ and $\beta \equiv \mu \pmod{2}$, let $m \ge 3$ be an odd integer,  and let $y$ be an integer with  $0 \le y \le \lfloor \frac{\mu m}{2} \rfloor-\lfloor \frac{\beta}{2} \rfloor$. Furthermore, let $T'_1,\ldots,T'_{y}$ be any bipartite 2-factor types that are admissible for $2K_n$ but not for $K_n$, and let $T_1,\ldots,T_{\lfloor \frac{\beta}{2} \rfloor}$
be any admissible bipartite 2-factor types for $2K_{n \times m}$ that are refinements of $[ (2m)^{\la \frac{n}{2} \ra} ]$.
\begin{enumerate}[(i)]
\item If $n=4$, then $\mu K_{n \times m}$ admits a 2-factorization of type
$\Big[ [4m]^{\la a \ra}, mT'_1,\ldots,mT'_{y}, T_1,\ldots,T_{\lfloor \frac{\beta}{2} \rfloor} \Big]$ for all non-negative integers $a$ such that $a+y=\lfloor \frac{\mu m(n-1)}{2} \rfloor - \lfloor \frac{\beta}{2} \rfloor$.
\item If $n=6$, then $\mu K_{n \times m}$ admits a 2-factorization of type
$\Big[ [6m]^{\la a \ra}, [3m,3m]^{\la b \ra}, mT'_1,\ldots,mT'_{y},$ $ T_1,\ldots,T_{\lfloor \frac{\beta}{2} \rfloor} \Big]$ for all  non-negative integers $a,b$ such that
    \begin{itemize}
    \item $a+b+y=\lfloor \frac{\mu m(n-1)}{2} \rfloor - \lfloor \frac{\beta}{2} \rfloor$, and
    \item $b \le \mu m$.
    \end{itemize}
\item If $n=8$, then $\mu K_{n \times m}$ admits a 2-factorization of type
$\Big[ [8m]^{\la a \ra}, [5m,3m]^{\la b \ra},[4m,4m]^{\la c \ra},$ $ mT'_1,\ldots,mT'_{y}, T_1,\ldots,T_{\lfloor \frac{\beta}{2} \rfloor} \Big]$ for all  non-negative integers $a,b,c$ such that
    \begin{itemize}
    \item $a+b+c+y=\lfloor \frac{\mu m(n-1)}{2} \rfloor - \lfloor \frac{\beta}{2} \rfloor$,
    \item $a+c+y \ge \lfloor \frac{\mu m}{2} \rfloor - \lfloor \frac{\beta}{2} \rfloor$, and
    \item if $a=0$, then $c+y \equiv \lfloor \frac{\mu m}{2} \rfloor - \lfloor \frac{\beta}{2} \rfloor \pmod{3} $.
    \end{itemize}
\item If $n=10$, then $\mu K_{n \times m}$ admits a 2-factorization of type
$\Big[ [10m]^{\la a \ra}, [7m,3m]^{\la b \ra},[6m,4m]^{\la c \ra},$ $[5m,5m]^{\la d \ra},[4m,3m,3m]^{\la e \ra}, mT'_1,\ldots,mT'_{y}, T_1,\ldots,T_{\lfloor \frac{\beta}{2} \rfloor} \Big]$ for all  non-negative integers $a$, $b,c,d,e$ such that
    \begin{itemize}
    \item $a+b+c+d+e+y=\lfloor \frac{\mu m(n-1)}{2} \rfloor - \lfloor \frac{\beta}{2} \rfloor$ and
    \item $a+c+y \ge \lfloor \frac{\mu m}{2} \rfloor - \lfloor \frac{\beta}{2} \rfloor$.
    \end{itemize}
\end{enumerate}
\end{cor}

\begin{proof}
Since $\beta \le \mu$ and $\beta \equiv \mu \pmod{2}$, we can write $\mu = \sum_{i=1}^{\beta} \mu_i$ for positive odd integers $\mu_i$ (for example, $\mu_1=\ldots=\mu_{\beta-1}=1$ and $\mu_{\beta}=\mu-\beta+1$). Observe that for all $i=1,\ldots,\beta$, since $\mu_i m$ is odd and $n$ is even,
$$\left\lfloor \frac{\mu_i m(n-1)}{2} \right\rfloor - \left\lfloor \frac{\mu_i m}{2} \right\rfloor = \frac{\mu_i m(n-1)-1}{2}-\frac{\mu_i m-1}{2}=\mu_i m \frac{n-2}{2}.$$
Moreover, since  $\beta \equiv \mu \pmod{2}$,
$$\left\lfloor \frac{\mu m(n-1)}{2} \right\rfloor - \left\lfloor \frac{\beta}{2} \right\rfloor =\frac{\mu m(n-1)-\beta}{2}=
\sum_{i=1}^{\beta} \frac{\mu_i m(n-1)-1}{2}=\sum_{i=1}^{\beta} \left\lfloor \frac{\mu_i m(n-1)}{2} \right\rfloor,$$
and similarly,
$$\left\lfloor \frac{\mu m}{2} \right\rfloor - \left\lfloor \frac{\beta}{2} \right\rfloor =\sum_{i=1}^{\beta} \left\lfloor \frac{\mu_i m}{2} \right\rfloor.$$

We shall use Theorem~\ref{the:layers2} with $\ell=\beta$ and all $x_i=1$.

\begin{enumerate}[(i)]
\item Assume $n=4$ and $a+y=\lfloor \frac{\mu m(n-1)}{2} \rfloor - \lfloor \frac{\beta}{2} \rfloor$. Since $0 \le y \le \lfloor \frac{\mu m}{2} \rfloor-\lfloor \frac{\beta}{2} \rfloor$ and $\lfloor \frac{\mu m}{2} \rfloor-\lfloor \frac{\beta}{2} \rfloor=\sum_{i=1}^{\beta} \left\lfloor \frac{\mu_i m}{2} \right\rfloor$, we can write $y=\sum_{i=1}^{\beta} y_i$, where each $y_i$ is a non-negative integer with $y_i \le \lfloor \frac{\mu_i m}{2} \rfloor$. For $i=1,\ldots,\beta$, we let $a_i=\lfloor \frac{\mu_i m(n-1)}{2} \rfloor - y_i$, so that $a=\sum_{i=1}^{\beta} a_i$. Relabel $T'_1,\ldots,T'_y$ so that $[T'_1,\ldots,T'_y]= \sqcup_{i=1}^{\beta} [T'_{i,1},\ldots,T'_{i,y_i}]$. By Corollary~\ref{cor:layers-small-even-n}(i), for each $i=1,\ldots,\beta$, since $a_i+y_i=\lfloor \frac{\mu_i m(n-1)}{2} \rfloor$,
    there exists a 2-factorization of $\mu_i m K_n$ of type $\Big[ [4]^{\la a_i \ra}, T'_{i,1},\ldots,T'_{i,y_i} \Big]$. Hence by Theorem~\ref{the:layers2}, $\mu K_{n \times m}$ admits a 2-factorization of type
    $$ \sqcup_{i=1}^{\beta} \Big[ [4m]^{\la a_i \ra}, mT'_{i,1},\ldots,mT'_{i,y_i} \Big] \sqcup
    \Big[ T_1,\ldots,T_{\lfloor \frac{\beta}{2} \rfloor} \Big]
    = \Big[ [4m]^{\la a \ra}, mT'_{1},\ldots,mT'_{y},T_1,\ldots,T_{\lfloor \frac{\beta}{2} \rfloor} \Big].$$

\item Assume $n=6$ and integers $a,b,y$ satisfy the conditions of this case. Similarly to Case (i), by Corollary~\ref{cor:layers-small-even-n}(ii) and Theorem~\ref{the:layers2}, it suffices to find non-negative integers $y_1,\ldots,y_{\beta}$, $a_1,\ldots,a_{\beta}$, and $b_1,\ldots,b_{\beta}$ such that $y=\sum_{i=1}^{\beta} y_i$, $a=\sum_{i=1}^{\beta} a_i$, and $b=\sum_{i=1}^{\beta} b_i$, and for each $i=1,\ldots,\beta$ the following hold: $ y_i \le \lfloor \frac{\mu_i m}{2} \rfloor$,
    $a_i+b_i+y_i=\lfloor \frac{\mu_i m(n-1)}{2} \rfloor$, and
    $b_i \le \mu_i m$. This is easy to accomplish: we first choose suitable $y_1,\ldots,y_{\beta}$ so that $y=\sum_{i=1}^{\beta} y_i$ and each $y_i \le \lfloor \frac{\mu_i m}{2} \rfloor$, then suitable $b_1,\ldots,b_{\beta}$ so that $b=\sum_{i=1}^{\beta} b_i$ and each $b_i \le \mu_i m$. For each $i$, we then let $a_i=\lfloor \frac{\mu_i m(n-1)}{2} \rfloor-b_i-y_i$, resulting in $a=\sum_{i=1}^{\beta} a_i$.

\item Assume $n=8$ and integers $a,b,c,y$ satisfy the conditions of this case. By Corollary~\ref{cor:layers-small-even-n}(iii) and Theorem~\ref{the:layers2}, it suffices to find non-negative integers $y_i$, $a_i$, $b_i$, and $c_i$, for $i=1,\ldots,\beta$, such that $y=\sum_{i=1}^{\beta} y_i$, $a=\sum_{i=1}^{\beta} a_i$, $b=\sum_{i=1}^{\beta} b_i$,  $c=\sum_{i=1}^{\beta} c_i$, and for each $i$ the following hold: $ y_i \le \lfloor \frac{\mu_i m}{2} \rfloor$,
    $a_i+b_i+c_i+y_i=\lfloor \frac{\mu_i m(n-1)}{2} \rfloor$,
    $a_i+c_i+y_i \ge \lfloor \frac{\mu_i m}{2} \rfloor$, and
    if $a_i=0$, then $c_i+y_i \equiv \lfloor \frac{\mu_i m}{2} \rfloor \pmod{3} $. This is achieved as follows. We first choose suitable $y_1,\ldots,y_{\beta}$ so that $y=\sum_{i=1}^{\beta} y_i$ and each $y_i \le \lfloor \frac{\mu_i m}{2} \rfloor$.

    {\sc Case $a \ge \beta$.} By the assumptions on $a$, $c$, and $y$, we can choose $z_1,\ldots,z_{\beta}$ such that $a+c=\sum_{i=1}^{\beta} z_i$ and
    $\max\{1, \lfloor \frac{\mu_i m}{2} \rfloor-y_i\} \le z_i \le \lfloor \frac{\mu_i m(n-1)}{2} \rfloor -y_i$ for all $i$. Then, for $i=1,\ldots,\beta$, we choose $a_i$ with $1\le a_i \le z_i$ so that $a=\sum_{i=1}^{\beta} a_i$. Finally, for each $i$, we let $c_i=z_i-a_i$ and $b_i=\lfloor \frac{\mu_i m(n-1)}{2} \rfloor-(a_i+c_i+z_i)$.

    {\sc Case $a=0$.} First, let $a_1=\ldots=a_{\beta}=0$. Next, choose $c_1,\ldots,c_{\beta}$ such that $c=\sum_{i=1}^{\beta} c_i$ and for all $i$, $\lfloor \frac{\mu_i m}{2} \rfloor -y_i \le c_i \le \lfloor \frac{\mu_i m(n-1)}{2} \rfloor -y_i$ and
    $c_i-(\lfloor \frac{\mu_i m}{2} \rfloor -y_i) \equiv 0  \pmod{3}$. The latter requirement can be fulfilled because
    $\lfloor \frac{\mu_i m(n-1)}{2} \rfloor -  \lfloor \frac{\mu_i m}{2} \rfloor \equiv 0 \pmod{3}$ and because, by assumption, $c-(\sum_{i=1}^{\beta} \lfloor \frac{\mu_i m}{2} \rfloor -y) \equiv 0 \pmod{3}$. Finally, for each $i$, let $b_i=\lfloor \frac{\mu_i m(n-1)}{2} \rfloor-(a_i+c_i+z_i)$.

    {\sc Case $1 \le a < \beta$.}
    Let $a_1=\ldots=a_{a}=1$ and $a_{a+1}=\ldots=a_{\beta}=0$.
    Next, choose non-negative integers $c'_1,\ldots,c'_{\beta}$ such that
    $\sum_{i=1}^{\beta} c'_i=a+c+y- \left( \lfloor \frac{\mu m}{2} \rfloor - \lfloor \frac{\beta}{2} \rfloor\right)$ and
    $c'_i \le \lfloor \frac{\mu_i m(n-1)}{2} \rfloor - \lfloor \frac{\mu_i m}{2} \rfloor$. This is possible since $a+c+y - \left( \lfloor \frac{\mu m}{2} \rfloor - \lfloor \frac{\beta}{2} \rfloor \right) \ge 0$ and
    \begin{eqnarray*}
    a+c+y - \left( \left\lfloor \frac{\mu m}{2} \right\rfloor - \left\lfloor \frac{\beta}{2} \right\rfloor \right) & \le & \left( \left\lfloor \frac{\mu m (n-1)}{2} \right\rfloor - \left\lfloor \frac{\beta}{2} \right\rfloor \right) - \left( \left\lfloor \frac{\mu m}{2} \right\rfloor - \left\lfloor \frac{\beta}{2} \right\rfloor \right) \\
    & = & \sum_{i=1}^{\beta} \left\lfloor \frac{\mu_i m(n-1)}{2} \right\rfloor - \sum_{i=1}^{\beta} \left\lfloor \frac{\mu_i m}{2} \right\rfloor.
    \end{eqnarray*}
    In addition, since $\lfloor \frac{\mu_i m(n-1)}{2} \rfloor - \lfloor \frac{\mu_i m}{2} \rfloor \equiv 0 \pmod{3}$ and $a \ge 1$, we can choose $c'_i \equiv 0 \pmod{3}$ for all $i=a+1,\ldots,\beta$. Moreover, we show that we can choose $c'_i \ge 1$ for all $i$ such that $a_i=1$ and $y_i=\left\lfloor \frac{\mu_i m}{2} \right\rfloor$. Let $\gamma$ be the number of such indices. We may assume that $y_1,\ldots,y_{\beta}$ were chosen so that $\gamma$ is minimum; that is, we have $y_i=\left\lfloor \frac{\mu_i m}{2} \right\rfloor$ for all $i \ge a+1$ and $y = \sum_{i=1}^{\beta} \left\lfloor \frac{\mu_i m}{2} \right\rfloor -(a-\gamma)$. Hence
    $$\sum_{i=1}^{\beta} c_i'=a+c+y - {\textstyle \left( \left\lfloor \frac{\mu m}{2} \right\rfloor - \left\lfloor \frac{\beta}{2} \right\rfloor
    \right)}
    = a+c+  \sum_{i=1}^{\beta} \left\lfloor \frac{\mu_i m}{2} \right\rfloor -a + \gamma -
    {\textstyle \left( \left\lfloor \frac{\mu m}{2} \right\rfloor - \left\lfloor \frac{\beta}{2} \right\rfloor
    \right)}
    = c+ \gamma \ge \gamma.$$
    Thus, indeed, $c_i' \ge 1$ can be chosen for the required $\gamma$ indices $i$.

    Now, for all $i=1,\ldots,\beta$, let $c_i=c'_i + \left\lfloor \frac{\mu_i m}{2} \right\rfloor-a_i-y_i$. Since $c'_i \ge 1$ whenever $a_i=1$ and $y_i=\left\lfloor \frac{\mu_i m}{2} \right\rfloor$, we conclude that $c_i \ge 0$.
Observe that if $a_i=0$, then $c_i+y_i=c'_i + \left\lfloor \frac{\mu_i m}{2} \right\rfloor \equiv \left\lfloor \frac{\mu_i m}{2} \right\rfloor \pmod{3}$.
    Moreover,
    $$a_i+c_i+y_i = a_i + (c'_i + \left\lfloor \frac{\mu_i m}{2} \right\rfloor-a_i-y_i) + y_i = c'_i + \left\lfloor \frac{\mu_i m}{2} \right\rfloor,$$
    so $\left\lfloor \frac{\mu_i m}{2} \right\rfloor \le a_i+c_i+y_i \le \left\lfloor \frac{\mu_i m(n-1)}{2} \right\rfloor$.
    Finally, let
    $b_i=\lfloor \frac{\mu_i m(n-1)}{2} \rfloor - (a_i+c_i+y_i)$.

Thus, in all cases, the required integers $y_i$, $a_i$, $b_i$, and $c_i$, for $i=1,\ldots,\beta$, can be found, and the result follows from Corollary~\ref{cor:layers-small-even-n}(iii) and Theorem~\ref{the:layers2}.

\item Assume $n=10$ and integers $a,b,c,d,e,y$ satisfy the conditions of this case. It is easy to see that there exist non-negative integers $y_i$, $a_i$, $b_i$, $c_i$, $d_i$, and $e_i$, for $i=1,\ldots,\beta$,
     such that $y=\sum_{i=1}^{\beta} y_i$, $a=\sum_{i=1}^{\beta} a_i$, $b=\sum_{i=1}^{\beta} b_i$, $c=\sum_{i=1}^{\beta} c_i$, $d=\sum_{i=1}^{\beta} d_i$, and $e=\sum_{i=1}^{\beta} e_i$,
     and for each $i$ the following hold: $ y_i \le \lfloor \frac{\mu_i m}{2} \rfloor$,
    $a_i+b_i+c_i+d_i+e_i+y_i=\lfloor \frac{\mu_i m(n-1)}{2} \rfloor$, and
    $a_i+c_i+y_i \ge \lfloor \frac{\mu_i m}{2} \rfloor$.
    The result then follows by Corollary~\ref{cor:layers-small-even-n}(iv) and Theorem~\ref{the:layers2},
\end{enumerate}
\end{proof}

\section{Final Remarks}

We first briefly mention another application of detachment, namely, to cycle frames.

A {\em $(L,\lambda)$-cycle frame of type $m^n$} is a holey 2-factorization of $\lambda K_{n \times m}$ such that each holey 2-factor consists of cycles of a length in the set $L$. Thus, a holey 2-factorization  of type $[T_1,\ldots,T_k]$ for the graph of $\lambda K_{n \times m}$ is an $(L,\lambda)$-cycle frame of type $m^n$, where $L=\{ \ell: \ell \in T_i \mbox{ for some } i \}$. For uniform cycle frames, we
write simply $(\ell,\lambda)$-cycle frame instead of $(\{\ell\},\lambda)$-cycle
frame. Note that an $(\ell,\lambda)$-cycle frame of type $m^n$ is equivalent to a
holey 2-factorization of $\lambda K_{n \times m}$ of type $[[\ell^\ast]^\ast]$. The
existence problem for uniform cycle frames and $(\{3,5\},\lambda)$-cycle frames of
type $m^n$ were completely settled in \cite{MR3632402} and \cite{MR2961229},
respectively.

By Corollary~\ref{cor:direct}(iv), if $\lambda m K_n$ admits an almost 2-factorization of type $[T_1,\ldots,T_k]$, then $\lambda K_{n \times m}$ admits a holey 2-factorization of type $[mT_1,\ldots,mT_k]$. Combining this statement with the result
of  \cite{MR1351903} on uniform almost 2-factorizations of  complete multigraphs, we obtain $(m\ell,\lambda)$-cycle frames of type $m^n$ whenever $\lambda m$ is even and $\ell$ divides $n-1$. This is not a new result, however,  its proof is much simpler than the proof of \cite{MR3632402}, which is heavily algebraic and
relies on many previous results.  It would be interesting to obtain other existence results on almost 2-factorizations of complete multigraphs, as these would immediately lead to new results on cycle frames.

To summarize --- in this paper, we developed a general framework for constructing new resolvable cycle decomposition of complete multigraphs and complete equipartite multigraphs by combining the techniques of detachment and layering. Numerous new resolvable cycle decompositions were obtained by applying this method to existing results. However, our approach does have some limitations.
Layering of complete graphs inevitably leads to 2-factorizations with 2-factor types inherited from seed 2-factorizations, plus additional arbitrary bipartite 2-factor types when the order is even. Even more restrictive is detachment;  applied to $\lambda m K_{n}$, it results in cycle decompositions of $\lambda K_{n \times m}$ with all cycles of length divisible by $m$. (The only exception is Theorem~\ref{the:layers2}, where detachment is applied to complete equipartite multigraphs and combined with layering.)
It would thus be desirable, and we think it might be possible, to modify our detachment technique so that it yields cycle decompositions of $\lambda K_{n \times m}$ without the restriction that the resulting cycle lengths all be multiples of $m$. We leave this idea as an open problem.

\begin{center}
{\large \bf Acknowledgement}
\end{center}

The  authors' research is  supported by NSA Grant H98230-16-1-0304 and NSERC Discovery Grant RGPIN-2016-04798, respectively.
\medskip

\footnotesize{\bibliographystyle{plain}

}

\end{document}